\documentclass{amsart}
\usepackage{graphicx}
\usepackage{amssymb}

\newtheorem{thm}{Theorem}[section]

\newtheorem{lem}[thm]{Lemma}
\newtheorem{prop}[thm]{Proposition}
\theoremstyle{definition}

\theoremstyle{remark}
\newtheorem{rem}[thm]{Remark}

\numberwithin{equation}{section}

\begin{document}

\title[Obstacle Problem]{An obstacle problem for conical deformations\\ of thin elastic sheets}
\author{Alessio Figalli}
\address{Department of Mathematics, ETH Z\"{u}rich, R\"{a}mistrasse 101, 8092 Z\"{u}rich, Switzerland}
\email{\tt alessio.figalli@math.ethz.ch}
\author{Connor Mooney}
\address{Department of Mathematics, ETH Z\"{u}rich, R\"{a}mistrasse 101, 8092 Z\"{u}rich, Switzerland}
\email{\tt connor.mooney@math.ethz.ch}

%\subjclass[2010]{}
%\keywords{}

% ----------------------------------------------------------------
\begin{abstract}
A developable cone (``d-cone") is the shape made by an elastic sheet when it is pressed at its center into a hollow cylinder by a distance $\epsilon$.
Starting from a nonlinear model depending on the thickness $h > 0$ of the sheet, we prove a $\Gamma$-convergence result as $h \rightarrow 0$ to a fourth-order obstacle
problem for curves in $\mathbb{S}^2$. We then describe the exact shape of minimizers of the limit problem when $\epsilon$ is small. In particular, we rigorously justify previous results in the physics literature.
\end{abstract}

\maketitle

%%%%%%%%%%%%%%%%%%%%%%%%%%%%%%%%%%%%%%%%%%%%%%%%%%%%%%%%%%%%%%%%%%%%%%%%%%%%%%%%%%%%%%%%%%%%%%%%%%%%%%%%%%%%
\section{Introduction}
If a thin elastic sheet is placed on top of a hollow cylinder and pressed down by a distance $\epsilon$ in the center, then
the resulting shape is roughly a developable cone (known as ``d-cone''). Experiments show that the sheet lifts from the cylinder on one single region (that is, it has ``one fold''), and that the angle subtended by
this fold is independent of $\epsilon$ for $\epsilon$ small (see \cite{CM}). In this paper we give a mathematically rigorous justification of these observations. 

We start with a nonlinear model depending on the thickness $h > 0$ of the sheet. We first prove a $\Gamma$-convergence result
as $h \rightarrow 0$ to a fourth-order obstacle problem for unit-speed curves in $\mathbb{S}^2$.
We then show that, for $\epsilon$ small, the minimizers of the one-dimensional obstacle problem lift from the obstacle on exactly one interval, and we give a precise estimate for the length of this interval. We describe our results
in more detail below.

\vspace{3mm}

We model the elastic energy of a sheet of thickness $h$ by
$$E_h({\bf u}) := h^2 \int_{B_1} |D^2{\bf u}|^2 \,dx + \int_{B_1} \text{dist}^2(D{\bf u},\, O(2,\,3))\,dx,$$
where ${\bf u}: B_1 \subset \mathbb{R}^2 \rightarrow \mathbb{R}^3$ is a $W^{2,2}$ map, and $O(2,\,3) = \{P \in M^{3 \times 2}: P^TP = I_{2 \times 2}\}$.
We restrict our attention to maps satisfying the conditions ${\bf u}(0) = 0$ and ${\bf u}|_{\mathbb{S}^1} \in \mathbb{S}^2$ (here and in the sequel, $\mathbb S^1=\partial B_1$ is the unit circle in $\mathbb R^2$). It is well-known that the minimizers of $E_h$ with these boundary conditions have the energy scaling $E_h \sim h^2|\log h|$
(\cite{BKN}, \cite{MO}). It is thus natural to consider the normalized energy
\begin{equation}\label{NormalizedEnergy}
\overline{E}_h := \frac{1}{h^2|\log h|}E_h.
\end{equation}

Our first result is that, subject to the above conditions,
the functionals $\overline{E}_h$ $\Gamma$-converge as $h \rightarrow 0$ to a limit functional on one-homogeneous isometries with $W^{2,2}$ boundary data. More precisely, let
$$X := \{{\bf u} \in W^{1,4}(B_1; \,\mathbb{R}^3): {\bf u}(0) = 0 \text{ and } {\bf u}|_{\mathbb{S}^1} \in \mathbb{S}^2\}$$
equipped with the $W^{1,4}$ norm. Let $\overline{E}_h$ the the functional on $X$ defined by (\ref{NormalizedEnergy}) when ${\bf u} \in X \cap W^{2,2}(B_1)$, and by $\overline{E}_h = +\infty$ otherwise.
Furthermore, define the functional $\overline{E}_0$ on $X$ by
\begin{equation}\label{LimitFunctional}
\overline{E}_0({\bf u}) :=\left\{ \begin{array}{ll}
\frac{1}{\log 2}\int_{B_1 \backslash B_{1/2}} |D^2{\bf u}|^2\,dx,\, &\text{if }{\bf u} \in X \cap W^{2,2}(B_1 \backslash B_{1/2})\\
& \quad\text{ is a one-homogeneous isometry}, \\
+\infty, &\text{otherwise}.
\end{array}
\right.
\end{equation}

The first main result is:
\begin{thm}\label{MainTheorem1}
The functionals $\overline{E}_h$ $\Gamma$-converge on $X$ to the functional $\overline{E}_0$.
\end{thm}

\begin{rem}
The pointwise convergence of $\overline{E}_h$
to $\overline{E}_0$ for fixed unit-speed Dirichlet data in $C^k(\mathbb{S}^1)$ for $k = 2$ or $k = 3$ follows from the estimates
in \cite{BKN} and \cite{MO}.
Here, to prove our $\Gamma$-convergence result, we only have the information ${\bf u}|_{\mathbb{S}^1} \in \mathbb{S}^2$.
\end{rem}

The main difficulty in the proof of Theorem \ref{MainTheorem1} is that the normalized energy does not control the $W^{2,2}$ norm of the boundary data. To overcome this,
we prove some geometric estimates to show that certain Lipschitz rescalings of maps with bounded normalized energy 
have boundary data that are bounded in $W^{2,2}$, and are close to the original data in $L^2$ (see Section \ref{Gamma_Convergence}).

\vspace{3mm}
To model a thin elastic sheet placed on top of a hollow cylinder and pressed down by a distance $\epsilon$ in the center,
we introduce the obstacle
$$O_{\epsilon} := \{x_1^2 + x_2^2 > 1-\epsilon^2\} \cap \{x_3 < \epsilon\},$$ 
and we let $X_{\epsilon} = X \cap \{{\bf u}(B_1) \cap O_{\epsilon} = \emptyset\}$ equipped with the $W^{1,4}$ norm. 
Define $\overline{E}_h$ and $\overline{E}_0$ on $X_{\epsilon}$ as above. The existence of minimizers of $\overline{E}_h$ in $X_{\epsilon} \cap W^{2,2}(B_1)$
and of $\overline{E}_0$ in $X_{\epsilon} \cap W^{2,2}(B_1 \backslash B_{1/2})$ is an easy consequence of the direct method in the Calculus of Variations.
A corollary of the proof of Theorem \ref{MainTheorem1} is that the functionals $\overline{E}_h$ on $X_{\epsilon}$ $\Gamma$-converge to $\overline{E}_0$ (see Remark \ref{ObstacleGammaConvergence}). 
In particular, if $h_k \rightarrow 0$ and minimizers ${\bf u}^{h_k}$ of $\overline{E}_{h_k}$ converge in $X_{\epsilon}$ to ${\bf u}^0$, then ${\bf u}^0$ is a minimizer in $X_{\epsilon}$ of $\overline{E}_0$.

\vspace{3mm}

Our second result is a precise description of the minimizers of $\overline{E}_0$ in $X_{\epsilon}$, for all $\epsilon$ small. If ${\bf u}$ is one such minimizer, then ${\bf u}|_{\mathbb{S}^1}$
is a unit-speed curve $\gamma \in W^{2,2}(\mathbb{S}^1; \,\mathbb{S}^2)$ with image in $\{x_3 \geq \epsilon\}$. Furthermore, $\overline{E}_0({\bf u}) = \int_{\mathbb{S}^1} \kappa^2\,ds$, where $\kappa$ is 
the geodesic curvature of $\gamma$. Thus, the problem of minimizing $\overline{E}_0$ in $X_{\epsilon}$ is equivalent to the fourth-order obstacle problem of minimizing
$$F(\gamma) := \int_{\mathbb{S}^1} \kappa^2\,ds$$
over $Y_{\epsilon} := \{\gamma \in W^{2,2}(\mathbb{S}^1; \,\mathbb{S}^2): |\gamma'| \equiv 1,\, \gamma \subset \{x_3 \geq \epsilon\}\}$.

\vspace{3mm}

Let $\gamma_{\epsilon}$ be a minimizer of $F$ in $Y_{\epsilon}$. Let $\theta$ denote the angular variable in cylindrical coordinates, with axis of symmetry in the $e_3$ direction. Our second result is:

\begin{thm}\label{MainTheorem2}
There exist $\epsilon_0 > 0$ small and $C$ universal such that for all $\epsilon < \epsilon_0$, we have
$$
\|\gamma_{\epsilon} \cdot e_3\|_{C^{2,1}(\mathbb{S}^1)} \leq C\epsilon,
$$
and $\gamma_{\epsilon}$ lifts from the obstacle $\{x_3 = \epsilon\} \cap \mathbb{S}^2$ on exactly one interval.

More precisely, $\gamma_{\epsilon}$ can be parametrized as
\begin{equation}
\label{eq:C21 eps}
\gamma_{\epsilon} = \{(1-\alpha(\theta)^2)^{1/2}(\cos(\theta) e_1 + \sin(\theta) e_2) + \alpha(\theta) e_3,\, \theta \in \mathbb{S}^1\},
\quad \|\alpha\|_{C^{2,1}(\mathbb{S}^1)} \leq C\epsilon,
\end{equation}
where $\alpha(\theta) = \epsilon$ on $\mathbb{S}^1 \backslash I_{\epsilon}$, where $I_{\epsilon}$ is an open interval satisfying $|I_{\epsilon}|\to \ell_0$ as $\epsilon \to 0$,
where $\ell_0 \in (2.42,2.43)$ is uniquely characterized. 
\end{thm}

The idea behind the proof of the result is to study a linearized obstacle problem for graphs over $\mathbb{S}^1$, obtained by ``stretching the picture vertically'' by the factor $\epsilon^{-1}$. Using analytic techniques we characterize
the minimizers of this linear problem as functions that lift from the obstacle (the constant function $1$) on exactly one interval, with a precise estimate for the length of this interval. To show that this behavior
passes to the minimizers of the nonlinear problem for $\epsilon$ small, we need some compactness. This is provided by the $C^{2,1}$ estimate  \eqref{eq:C21 eps}, which is uniform in $\epsilon$. 
This $C^{2,1}$ estimate comes from a careful analysis combining the Euler-Lagrange ODE with energy minimality. 

\begin{rem}
In \cite{O}, the $\Gamma$-convergence as $\epsilon \rightarrow 0$ of the ``vertically stretched'' nonlinear obstacle problem to the linearized problem is established, and minimizers
of the linear problem are studied. In contrast, Theorem \ref{MainTheorem2} describes the exact shape of minimizers for the nonlinear obstacle problem for all $\epsilon$ small. With respect to \cite{O}, the new contributions
of this paper are:\\
- a sharper characterization of minimizers of the linear problem as functions which lift from the obstacle on exactly one region;\\
- the $C^{2,1}$ estimate  \eqref{eq:C21 eps} and a uniform lower bound on the separation regions (see Proposition \ref{BumpLength}),
which allow us to pass this result to the minimizers of the nonlinear problem.
\end{rem}

\begin{rem}
Although $\gamma_{\epsilon}$ is contained in the thin strip $\{\epsilon \leq x_3 \leq C\epsilon\} \cap \mathbb{S}^2$ (so after ``stretching vertically'' we obtain graphs on a cylinder), 
the curvature of $\mathbb{S}^2$ plays an important role in making $\gamma_{\epsilon}$ stick to the obstacle.
Indeed, due to this constraint, the contributions of the height $\gamma_{\epsilon} \cdot e_3$ and its second derivative $(\gamma_{\epsilon} \cdot e_3)''$ to the curvature of $\gamma_{\epsilon}$ are of the same order for all 
$\epsilon$ small (see Section \ref{MinimizerDescription}).
\end{rem}

\vspace{3mm}

The paper is organized as follows. In Section \ref{Gamma_Convergence} we establish some preliminary geometric estimates, and use them to prove Theorem \ref{MainTheorem1}. In Section \ref{MinimizerDescription} we
prove Theorem \ref{MainTheorem2}, in several steps. We first derive the Euler-Lagrange equation for $\gamma_{\epsilon}$, and we show that $\gamma_{\epsilon} \in C^{2,1}$.
We then prove the bound $\|\gamma_{\epsilon} \cdot e_3\|_{C^{2,1}(\mathbb{S}^1)} \leq C\epsilon$. Next we describe minimizers to the linearized problem. Finally, we combine this analysis with the $C^{2,1}$ estimate to prove Theorem \ref{MainTheorem2}.
In the Appendix we collect some calculations and results from functional analysis used to derive the Euler-Lagrange equation.

%%%%%%%%%%%%%%%%%%%%%%%%%%%%%%%%%%%%%%%%%%%%%%%%%%%%%%%%%%%%%%%%%%%%%%%%%%%%%%%%%%%%%%%%%%%%%%%%%%%%%%%%%%%%
\section{$\Gamma$-Convergence}\label{Gamma_Convergence}

In this section we prove Theorem \ref{MainTheorem1}. We first establish some geometric estimates for maps with bounded normalized energy.

\vspace{3mm}

\subsection{Geometric Estimates}
Let ${\bf u}$ be a $W^{2,2}$ map such that ${\bf u}(0) = 0$ and ${\bf u}$ has boundary data
$$\gamma(\theta) := {\bf u}(1,\,\theta) \in \mathbb{S}^2.$$
(Here and below we use standard polar coordinates $(r,\,\theta) \in [0,\infty)\times \mathbb S^1$). Let
$${\bf v}(r,\,\theta):=  r\gamma(\theta)$$
be the cone over the boundary values. Finally, let
$${\bf e} := {\bf u} - {\bf v}.$$
Here and in the following, we shall use subscripts to denote partial derivatives (for instance ${\bf u}_r=\partial_r {\bf u}$).

The key estimate is the following:

\begin{prop}\label{GeometricEstimate}
Let ${\bf u}$ be as above, and assume for some $C_0 \geq 1$ that $\overline{E}_h({\bf u}) \leq C_0$. Then
\begin{equation}\label{H1Estimate}
\int_{B_1 \backslash B_h} \frac{|{\bf e}_r|^2}{r}\,dx \leq 3 C_0^{1/2} h|\log h|,
\end{equation}
and furthermore
\begin{equation}\label{L2Estimate}
\left(\frac{1}{r} \int_{\partial B_r} |{\bf e}|^2\,ds\right)^{1/2} \le 3 C_0^{1/4} \left(\frac{h|\log h|}{r}\right)^{1/2}r
\end{equation}
for all $h \leq h_0$ universal and $r \geq h$.
\end{prop}

Inequality (\ref{L2Estimate}) says that  maps with bounded normalized energy are well-approximated by the cones over their boundary data at scales $r >> h|\log h|$. The proof of this fact relies only on
the bound for the stretching energy. Inequality (\ref{H1Estimate}) says that, on average, the radial derivative of ${\bf u}$ is close to $\gamma$ on circles.

Before proving Proposition \ref{GeometricEstimate} we need some preliminary inequalities.

\begin{lem}\label{StretchingInequality}
Let $f: [0,\,h] \rightarrow \mathbb{R}^n$ be a Lipschitz function satisfying $\frac{|f(h) - f(0)|}{h} = A \geq 2.$ Then
$$\int_0^h x(|f'|^2 - 1)^2 \,dx \geq h^2 \frac{A^4}{16}.$$
\end{lem}
\begin{proof}
Using the rescaling $f(x) \rightarrow \frac{1}{h}f(hx)$, we may assume that $h = 1$. By Cauchy-Schwarz we have
$$\int_0^1 x(|f'|^2-1)^2 \,dx \geq \left(\int_0^1 x^{1/2}|f'|^2\,dx - \int_0^1 x^{1/2}\,dx\right)^2=\left(\int_0^1 x^{1/2}|f'|^2\,dx - \frac{2}{3}\right)^2$$
and
$$\int_0^1 x^{1/2}|f'|^2\,dx \geq \frac{1}{2} \left| \int_0^1 f'\,dx \right|^2 = \frac{1}{2} A^2.$$
Combining these we obtain
$$\int_0^1 x(|f'|^2-1)^2 \,dx \geq \left(\frac{1}{2}A^2 - \frac{2}{3}\right)^2.$$
\end{proof}

As a consequence of Lemma \ref{StretchingInequality} we control the oscillation of bounded-energy maps at small scales:

\begin{lem}\label{OscillationEstimate}
Assume that $\overline{E}_h({\bf u}) \leq C_0.$ Then
$$\sup_{\partial B_h(x)} |{\bf u} - {\bf u}(x)| \leq 4C_0^{1/4}h|\log h|^{1/4}$$
for any $x \in B_{1-2h}$.
\end{lem}

\begin{proof}
By translating and adding a constant vector we may assume that $x = 0$ and ${\bf u}(0) = 0$. Let $M>0$ to be fixed later, and define
$$E_M := \{\theta\in \mathbb S^1: |{\bf u}(h,\,\theta)| > M h|\log h|^{1/4}\}.$$
By the upper energy bound we have
$$\int_{E_M}\int_0^h r\left(|{\bf u}_r|^2 - 1\right)^2\,dr\,d\theta \leq C_0h^2|\log h|.$$
On the other hand, by Lemma \ref{StretchingInequality} we have
$$\int_{E_M}\int_0^h r\left(|{\bf u}_r|^2 - 1\right)^2\,dr\,d\theta > \mathcal H^1(E_M)\, h^2|\log h| \, \frac{M^4}{16},$$
where $\mathcal H^1$ denotes the 1-dimensional Hausdorff measure.
We conclude from the previous inequalities that
\begin{equation}
\label{eq:M contr}
\mathcal H^1(E_M) < 16 C_0 M^{-4}.
\end{equation}
Assume by way of contradiction that
\begin{equation}
\label{eq:u contr}
|{\bf u}(h,\,0)| > 2Mh|\log h|^{1/4},
\end{equation}
and consider the half-circle
$$
\partial B_h^-:=\{(h,\theta):\theta \in (\pi/2,3\pi/2)\}.
$$
We repeat the above argument with $(h,\,0)$ in place of the origin, $\theta \in (3\pi/4,5\pi/4)$, and $r \in[0,\rho(\theta)]$
with $\rho(\theta):=2|\cos\theta|h \in (\sqrt{2}h,2h)$, so that
$$
(h,0)+(\rho(\theta),\theta) \in \partial B_h^-\qquad \forall\, \theta \in (3\pi/4,5\pi/4)
$$
(see Figure \ref{GeomEst_pic}). 
 \begin{figure}
 \centering
    \includegraphics[scale=0.25]{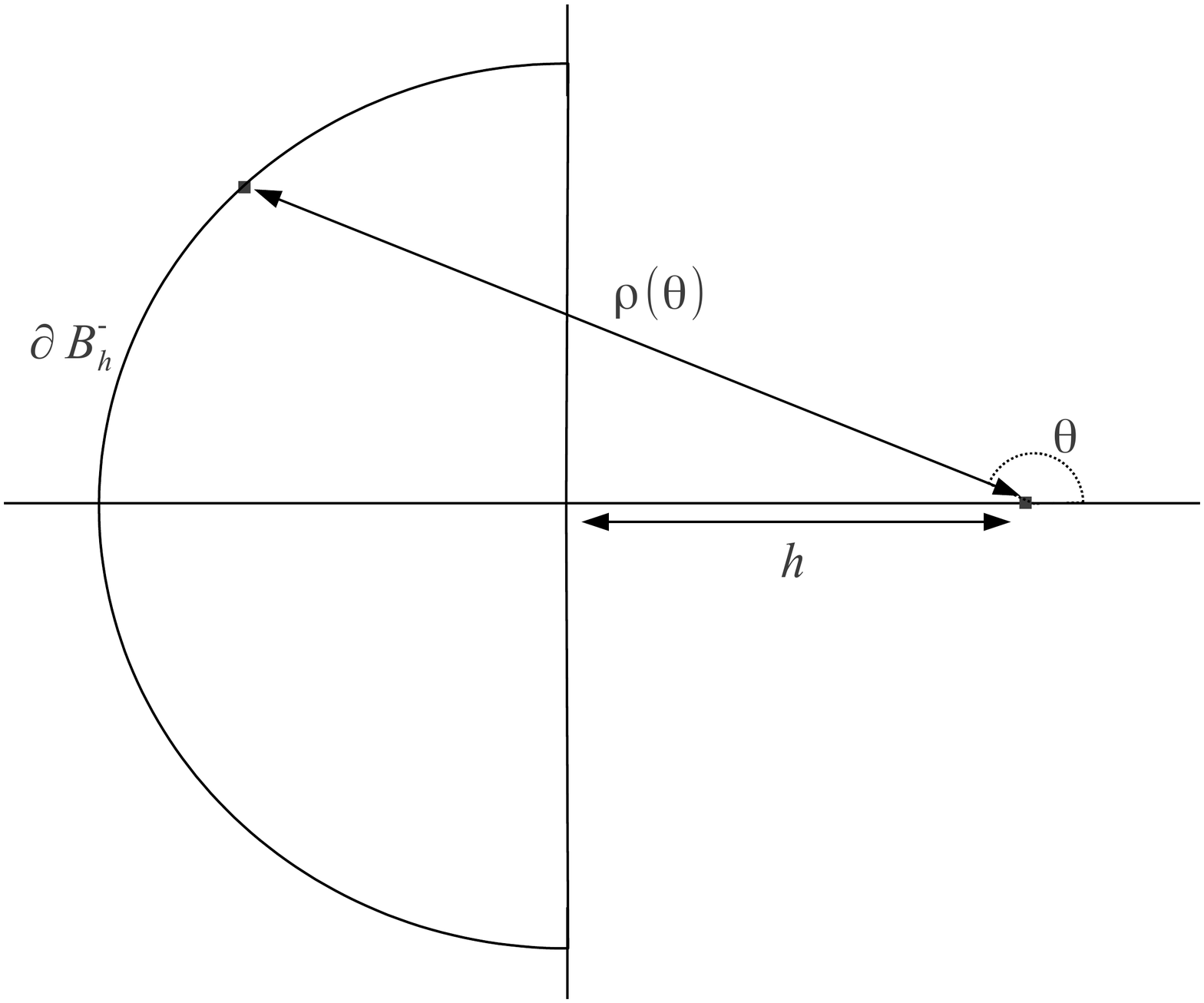}
 \caption{}
 \label{GeomEst_pic}
\end{figure}
In this way,
if we set
$$\tilde E_M := \bigl\{\theta\in (3\pi/4,5\pi/4): \bigr|{\bf u}\bigl((h,0)+(\rho(\theta),\theta)\bigr)-{\bf u}(h,\,0)\bigr| > M h|\log h|^{1/4}\bigr\},$$
 it follows by the upper bound and Lemma \ref{StretchingInequality} again, that
\begin{align*}
C_0h^2|\log h| &\geq \int_{\tilde{E}_M} \int_0^{\rho(\theta)} r(|{\bf u}_r|^2 - 1)^2\,dr\,d\theta \\
&> \mathcal H^1(\tilde E_M) (\sqrt{2}h)^2 |\log h|\frac{M^4}{16},
\end{align*}
thus
\begin{equation}
\label{eq:M contr2}
\mathcal H^1(\tilde E_M) < 8 C_0 M^{-4}.
\end{equation}
Recalling \eqref{eq:u contr}
it follows that the sets
$$
A_M:=\{(h,\theta):\theta \in (\pi/2,3\pi/2)\setminus E_M\}
$$
and
$$
\tilde A_M:=\{(h,0)+(\rho(\theta),\theta):\theta \in (3\pi/4,5\pi/4)\setminus \tilde E_M\},
$$
are disjoint subsets of $\partial B_h^-$, thus 
$$
\mathcal{H}^1(A_M)+\mathcal{H}^1(\tilde A_M)\leq \mathcal H^1(\partial B_h^-)=\pi h.
$$
On the other hand, \eqref{eq:M contr} and \eqref{eq:M contr2} imply that
$$
\mathcal{H}^1(A_M)=h\int_{(\pi/2,3\pi/2)\setminus E_M} d\theta> \bigl(\pi-16C_0M^{-4}\bigr)h, 
$$
\begin{align*}
\mathcal{H}^1(\tilde A_M)&=\int_{(3\pi/4,5\pi/4)\setminus \tilde E_M} \sqrt{\rho'(\theta)^2+\rho(\theta)^2}\,d\theta\\
&=
2h\int_{(3\pi/4,5\pi/4)\setminus \tilde E_M} d\theta
> \bigl(\pi-16C_0M^{-4}\bigr)h.
\end{align*}
Combining the last three estimates, we conclude that
$$
32 C_0M^{-4}> \pi,
$$
a contradiction if we chose $M:=2C_0^{1/4}$. 

This proves that \eqref{eq:u contr} is false, and since $(h,0)$ was an arbitrary point on $\partial B_h$, this concludes the proof.
\end{proof}

Now, using the boundary data, we prove Proposition \ref{GeometricEstimate}. The approach is the same as in \cite{BKN} and \cite{MO}.

\begin{proof}[{\bf Proof of Proposition \ref{GeometricEstimate}:}]
By the definition of ${\bf e}$ we have
$$(|{\bf u}_r|^2 - 1) - 2({\bf e}\cdot \gamma)_r = |{\bf e}_r|^2.$$
Using the boundary data and Lemma \ref{OscillationEstimate} we conclude that
$$\int_{B_1 \backslash B_h} \frac{|{\bf e}_r|^2}{r}\,dx \leq  \int_{B_1 \backslash B_h} \frac{(|{\bf u}_r|^2 - 1)}{r}\,dx + 16\pi C_0^{1/4} h|\log h|^{1/4}.$$
Applying Cauchy-Schwarz to the first term and using the energy bound we obtain the $H^1$ estimate
$$\int_{B_1 \backslash B_h} \frac{|{\bf e}_r|^2}{r}\,dx \leq 3C_0^{1/2} h|\log h|.$$
for $h < h_0$ universal. This is inequality (\ref{H1Estimate}).

This estimate controls the function $|{\bf e}|$ on circles. Indeed, by the fundamental theorem of calculus, for $r > h$ we have
$$\frac{1}{2}|{\bf e}|^2(r,\,\theta) \leq |{\bf e}|^2(h,\,\theta) + r \left(\int_{h}^1 |{\bf e}_r(\rho,\,\theta)|^2\,d\rho\right).$$
Integrating on $\partial B_r$ and using Lemma \ref{OscillationEstimate} and inequality (\ref{H1Estimate}) we obtain
$$\frac{1}{r}\int_{\partial B_r} |{\bf e}|^2 \,ds < 64\pi C_0^{1/2} h^2|\log h|^{1/2} + 6C_0^{1/2}\left(\frac{h|\log h|}{r}\right)r^2.$$
Since the second term dominates for $r \geq h$ whenever $h$ is sufficiently small, (\ref{L2Estimate}) follows.
\end{proof}

\vspace{3mm}

%%%%%%%%%%%
\subsection{$\Gamma$-Convergence}
We now prove the $\Gamma$-convergence of the functionals $\overline{E}_h$ on $X = \{{\bf u} \in W^{1,4}(B_1; \,\mathbb{R}^3): {\bf u}(0) = 0 \text{ and } {\bf u}|_{\mathbb{S}^1} \in \mathbb{S}^2\}$
to a limiting functional on conical isometries. Let ${\bf u}^h$ be a family of maps in $X$ such that $\overline{E}_h({\bf u}_h) \leq C_0$ for some $C_0 \geq 1$. The key result is the lower-semicontinuity:

\begin{prop}\label{LSC}
Under the above hypotheses, there exist $h_k \rightarrow 0$ such that ${\bf u}^{h_k}$ converge in $L^2$ to a
one-homogeneous isometry ${\bf u}^0 \in W^{2,2}(B_1 \backslash B_{1/2})$, and furthermore
$$\liminf_{h \rightarrow 0} \overline{E}_h({\bf u}^h) \geq \overline{E}_0({\bf u}^0).$$
\end{prop} 

As mentioned in the introduction, a difficulty of the proof is that the boundedness of the normalized energies does not imply the boundedness of  the boundary data ${\bf u}^h(1,\,\theta)$ in $W^{2,2}(\mathbb{S}^1)$. Consider for example
$${\bf w}^h(r,\,\theta) = r(\cos(\theta)\, e_1 + \sin(\theta)\, e_2) + \varphi(r)|\log h|^{1/2}\epsilon_h^2\sin(\theta/\epsilon_h)\,e_3,$$
where $\varphi$ is a smooth cutoff that is $1$ near $r = 1$ and $0$ for $r < 1/2$, and let
$${\bf u}^h (r,\,\theta) = \frac{{\bf w}^h(r,\theta)}{|{\bf w}^h(1,\,\theta)|}.$$
Then $\|D^2 {\bf u}^h\|_{L^2(B_1)}|\log h|^{-1/2}$ is bounded. Furthermore, the second term in the definition of ${\bf w}^h$ gets arbitrarily small in $C^1$ as $\epsilon_h \rightarrow 0$,
so for $\epsilon_h$ small we can keep the energy of ${\bf u}^h$ bounded. However, $\| {\bf w}^h \|_{W^{2,2}(\mathbb{S}^1)} \sim |\log h|^{1/2}$ blows up as $h \rightarrow 0$.

To overcome this difficulty, we use the geometric estimates to select some suitable Lipschitz rescalings of ${\bf u}^h$ whose boundary data have the same $L^2$ limit, but are bounded in $W^{2,2}$.

\begin{proof}[{\bf Proof of Proposition \ref{LSC}}]
Set $\gamma^h(\theta):={\bf u}^h(1,\theta)$, and define
$$R_h := \left\{r \in [h,\,1]: r^{-1}\int_{\partial B_r} |{\bf u}^h_r - \gamma^h|^2\,ds > |\log h|^{-1}\right\}.$$ 
By inequality (\ref{H1Estimate}) we have the estimate
\begin{equation}\label{RadialDerivativeEstimate}
\mathcal H^1(R_h) \leq Ch|\log h|^2.
\end{equation}
(Here and below $C$ denotes a constant depending on $C_0$). This implies that
$$
\int_{[h|\log h|^2, \,1]\backslash R_h} \frac{1}{r}\,dr= (1 + o(1))|\log h|,
$$
therefore
\begin{align*}
&\overline{E}_h({\bf u}^h) \\
&\geq |\log h|^{-1} \int_{[h|\log h|^2, \,1]\backslash R_h} \frac{1}{r^2} \int_{\partial B_r} \left(|rD^2{\bf u}^h|^2 + \frac{r^2}{h^2}\text{dist}^2(D{\bf u}^h,\,O(2,\,3))\right)\,ds\,dr \\
&\geq (1 + o(1))\inf_{[h|\log h|^2,\,1] \backslash R_h} \frac{1}{r}\int_{\partial B_r}  \left(|rD^2{\bf u}^h|^2 + \frac{r^2}{h^2}\text{dist}^2(D{\bf u}^h,\,O(2,\,3))\right)\,ds.
\end{align*}
Thus, we can choose $r_h \in [h|\log h|^2,\,1] \backslash R_h$ such that
\begin{equation}\label{LSCReduction}
\liminf_{h \rightarrow 0} \overline{E}^h({\bf u}^h) \geq \liminf_{h \rightarrow 0} \frac{1}{r_h} \int_{\partial B_{r_h}} |r_hD^2{\bf u}^h|^2 \,ds,
\end{equation}
\begin{equation}\label{IsometryControl}
\frac{1}{r_h}\int_{\partial B_{r_h}} \text{dist}^2(D{\bf u}^h, \,O(2,\,3))\,ds \leq C|\log h|^{-4},
\end{equation}
and
\begin{equation}\label{DerivativeControl}
\frac{1}{r_h}\int_{\partial B_{r_h}} |{\bf u}^h_r - \gamma^h|^2\,ds \leq |\log h|^{-1}.
\end{equation}
Indeed, (\ref{IsometryControl}) follows from the fact that $r_h \geq h|\log h|^2$,
while for (\ref{DerivativeControl}) 
we use that $r_h \notin R_h$.

Note that, by inequality (\ref{L2Estimate}), we also have
\begin{equation}\label{L2Convergence}
\frac{1}{r_h} \int_{\partial B_{r_h}} |r_h^{-1}{\bf u}^h - \gamma^h|^2\,ds \leq C|\log h|^{-1}.
\end{equation}
Set
$$\bar{\gamma}^h(\theta) := \frac{1}{r_h}{\bf u}^h(r_h,\,\theta).$$
We first claim that 
$$\|r\bar{\gamma}^h\|_{W^{2,2}(B_1 \backslash B_{1/2})} \leq C.$$
Indeed, by (\ref{L2Convergence}) we have $\|\gamma^h - \bar{\gamma}^h\|_{L^2(\mathbb{S}^1)} \leq C|\log h|^{-1/2}$. In addition, by (\ref{LSCReduction}) and (\ref{IsometryControl}) one obtains
$$\int_{\mathbb{S}^1} |\bar{\gamma}^h_{\theta\theta}|^2\,d\theta \leq C.$$
This proves the claim, and we conclude that (up to taking a subsequence) $\bar{\gamma}^h$ converge weakly in $W^{2,2}(\mathbb{S}^1)$ and strongly in $W^{1,2}(\mathbb{S}^1)$ to a $W^{2,2}$ limit curve $\gamma^0$. 
Note that, as a consequence of the $L^2$ convergence of $\gamma^h$ to $\gamma^0$,
and of (\ref{IsometryControl}) and the strong $W^{1,2}$ convergence of $\bar{\gamma}^h$ to $\gamma^0$, we have 
$$|\gamma^0| = |\gamma^0_{\theta}| \equiv 1.$$

Finally, we can establish lower semicontinuity. 
Since the matrix $D^2{\bf u}^h$ contains $\frac{1}{r^2}{\bf u}_{\theta\theta}^h
+\frac{1}{r}{\bf u}^h_r$ as one of its coefficients
(this follows by computing the Hessian in polar coordinates), we have
\begin{align*}
&\liminf_{h \rightarrow 0} \overline{E}_h({\bf u}^h) \geq \liminf_{h \rightarrow 0} \frac{1}{r_h}\int_{\partial B_{r_h}} |r_hD^2{\bf u}^h|^2\,ds \\
&\geq \liminf_{h \rightarrow 0} \int_{\mathbb{S}^1} |\bar{\gamma}^h_{\theta\theta} + {\bf u}^h_r(r_h,\, \cdot)|^2\,d\theta \\
&\geq \liminf_{h \rightarrow 0} \left( (1-\delta) \int_{\mathbb{S}^1} |\bar{\gamma}^h_{\theta\theta} + \bar{\gamma}^h|^2 \,d\theta - C(\delta)\int_{\mathbb{S}^1} (|{\bf u}^h_r(r_h, \,\cdot) - \gamma^h|^2 + |\gamma^h - \bar{\gamma}^h|^2)\,d\theta \right) \\
&\geq (1-\delta) \int_{\mathbb{S}^1} |\gamma^0_{\theta\theta} + \gamma^0|^2\,d\theta.
\end{align*}
In the last line we use the lower semicontinuity of the $W^{2,2}$ norm, and inequalities (\ref{DerivativeControl}) and (\ref{L2Convergence}).
Since this holds for all $\delta > 0$, we conclude that
\begin{equation}\label{LowerSemicontinuity}
\liminf_{h \rightarrow 0} \overline{E}_h({\bf u}_h) \geq \int_{\mathbb{S}^1} |\gamma^0_{\theta\theta} + \gamma^0|^2\,d\theta=\frac{1}{\log 2} \int_{B_1 \backslash B_{1/2}} |D^2{\bf u}^0|^2\,dx,
\end{equation}
where ${\bf u}^0(r,\, \theta) := r\gamma^0(\theta)$ is a one-homogeneous $W^{2,2}$ isometry that coincides with the $L^2$ limit of ${\bf u}^h$.
\end{proof}

%%%%%%
Theorem \ref{MainTheorem1} follows quickly from Proposition \ref{LSC}.

%\begin{rem}
%The Dirichlet conditions in the definition of $X$ are necessary. If we remove that the boundary values are in $\mathbb{S}^2$, look at isometries with bounded curvature passing through the origin. 
%If we remove the condition at the origin, it seems likely that there are isometries with bounded curvature that have boundary data in $\mathbb{S}^2$ but don't pass through the origin.
%\end{rem}

\begin{proof}[{\bf Proof of Theorem \ref{MainTheorem1}}]
We first show the lower semicontinuity inequality.

Assume that ${\bf u}^h$ converge to ${\bf u}^0$ in $X$. In the case
$\liminf_{h \rightarrow 0} \overline{E}_h({\bf u}^h) = +\infty$ we are done. In the other case, Proposition \ref{LSC} shows that ${\bf u}^0 \in W^{2,2}(B_1 \backslash B_{1/2})$ is a one-homogeneous isometry,
and furthermore that the lower semicontinuity inequality is satisfied.

Now assume that ${\bf u}^0 \in X$. We construct a recovery sequence.

If $\overline{E}_0({\bf u}^0 )= +\infty$ there is nothing to prove, so assume that ${\bf u}^0 \in W^{2,2}(B_1 \backslash B_{1/2})$ is a one-homogeneous isometry $r\gamma^0(\theta)$.
Let $f:\mathbb R\to\mathbb R$ be a smooth even function  such that $f(s) = |s|$ for $|s| \geq 1$ and $f(s)=s^2$ for $|s|\leq 1/2$.
Then 
$$|f''|+ |(f/s)'|+ |f'/s|+ |f/s^2| \leq C\qquad \forall\,s \in \mathbb R.
$$
Let ${\bf u}^h(r,\,\theta) := hf(r/h)\gamma^0(\theta)$. One computes
$$\overline{E}_h({\bf u}^h) = \overline{E}_0({\bf u}^0) + \frac{1}{|\log h|}\left( \int_{B_h} |D^2{\bf u}_h|^2\,dx + h^{-2} \int_{B_h} \text{dist}^2(D{\bf u}^h,\,O(2,\,3))\,dx\right).$$
By the inequalities for $f$, we have $|D^2{\bf u}^h| = O(h^{-1})$ and $|D{\bf u}^h| = O(1)$ in $B_h$. Thus, the last two terms are $O(|\log h|^{-1})$, and we conclude that
$$\lim\sup_{h \rightarrow 0} \overline{E}_h({\bf u}^h) = \overline{E}_0({\bf u}^0).$$
\end{proof}

\begin{rem}\label{ObstacleGammaConvergence}
Recall that $X_{\epsilon} = X \cap \{{\bf u}(B_1) \cap O_{\epsilon} = \emptyset\}$, where $O_{\epsilon}$ is the cylindrical obstacle $\{x_1^2 + x_2^2 > 1-\epsilon^2\} \cap \{x_3 < \epsilon\}$.
The same arguments show that the functionals $\overline{E}_h$ $\Gamma$-converge on $X_{\epsilon}$ to the functional $\overline{E}_0$.
To see this, note that $X_{\epsilon}$ is closed by the compact embedding of $W^{1,4}$ into $C^{0}$, that lower semicontinuity (Proposition \ref{LSC}) follows from bounded normalized energy,
and that the recovery sequence we construct above is in $X_{\epsilon}$ whenever ${\bf u}^0 \in X_\epsilon$.
\end{rem}

%%%%%%%%%%%%%%%%%%%%%%%%%%%%%%%%%%%%%%%%%%%%%%%%%%%%%%%%%%%%%%%%%%%%%%%%%%%%%%%
\section{Minimizers of the Limit Problem}\label{MinimizerDescription}
In this section we precisely describe the minimizers in $X_{\epsilon}$ of the limit energy $\overline{E}_0$.
Recall that this problem is equivalent to minimizing the so called ``Euler-Bernoulli elastica energy''
$$F(\gamma) := \int_{\mathbb{S}^1} \kappa^2\,ds$$
for curves in $Y_{\epsilon} = \{\gamma \in W^{2,2}(\mathbb{S}^1;\,\mathbb{S}^2): |\gamma'| = 1,\, \gamma \subset \{x_3 \geq \epsilon\}\}$. Here $\kappa = \gamma'' \cdot (\gamma \times \gamma')$ is the geodesic curvature of $\gamma$.

Note that all curves in $Y_\epsilon$ have length equal to $2\pi$.
Since $F(\gamma)$ is a geometric functional
(thus invariant under reparameterization), it can be defined also for general curves $\gamma$ (without the unit-speed constraint) with the understanding that $ds=ds_\gamma$ denotes the length element, and $\kappa$ is given by the formula $\kappa= \frac{\gamma'' \cdot (\gamma \times \gamma')}{|\gamma'|^3}$.
Hence, in order to have more freedom in our variations, we shall minimize $F$ over the set
$$
\tilde Y_{\epsilon} := \{\gamma \in W^{2,2}(\mathbb{S}^1;\,\mathbb{S}^2): {\rm Length}(\gamma) = 2\pi,\, \gamma \subset \{x_3 \geq \epsilon\}\},
$$
 where ${\rm Length}(\gamma):=\int ds_\gamma$ is the length functional.
%%%%%%%%%%%%%%%%%%%%%%%%%%%%%%%%%%%%%%%%%%%%%%%%%%%%%%%%%%
\subsection{Euler-Lagrange Equation}
Let $\gamma$ be a minimizer of $F$ in $\tilde Y_{\epsilon}$. Up to a reparameterization, we can assume that $|\gamma'|\equiv 1$, thus $\gamma \in Y_\epsilon$.
We first compute the Euler-Lagrange equation, and show that $\gamma\in C^{2,1}$:

\begin{lem}\label{EulerLagrange}
Let $\gamma$ be a unit-speed minimizer
of $F$ in $\tilde Y_\epsilon$. 
Then $\gamma\in C^{2,1}(\mathbb{S}^1)$, and for some $\lambda \in \mathbb{R}$, the geodesic curvature $\kappa$ satisfies
\begin{equation}\label{ODE}
\kappa'' + ((1+ \lambda) + \kappa^2/2)\kappa \geq 0,
\end{equation}
with equality where $\gamma \cdot e_3 > \epsilon$. Moreover, $F(\gamma) \leq C\epsilon^2$
and $\gamma \subset \{x_3 \leq C\epsilon\}$ for some universal $C$.
\end{lem}

Before proving Lemma \ref{EulerLagrange} we record some important variational inequalities.
Let $\varphi:\mathbb{S}^1 \rightarrow \mathbb{R}^3$ be a smooth map, and let $\psi := \varphi - (\varphi \cdot \gamma)\gamma$ be its projection tangent to the sphere at $\gamma$.
A calculation (see Appendix) gives
$$F\left(\frac{\gamma + \delta\psi}{|\gamma + \delta\psi|}\right) = F(\gamma) + 2\delta \int_{\mathbb{S}^1}  \left(\kappa N \cdot \psi'' - \frac{3}{2} \kappa^2 \gamma' \cdot \psi' + \kappa N \cdot \psi\right)\,ds + O(\delta^2),$$
where $N = \gamma \times \gamma'$ is the unit normal to the cone over $\gamma$.

By minimality, the first-order coefficient in $\delta$ is nonnegative provided the variation satisfies $\psi \cdot e_3 \geq 0$ where $\gamma$ touches the obstacle $\{x_3 = \epsilon\} \cap \mathbb{S}^2$ (this is needed to ensure that also $\frac{\gamma + \delta\psi}{|\gamma + \delta\psi|}$ is contained inside the set $\{x_3\geq \epsilon\}$ for $\delta\geq 0$), and preserves length to first order.
We can remove the length constraint
$$
\frac{d}{d\delta}|_{\delta=0}{\rm Length}\left(\frac{\gamma + \delta\psi}{|\gamma + \delta\psi|}\right)=\int_{\mathbb{S}^1} \kappa N \cdot \psi \,ds = 0
$$ by introducing a Lagrange multiplier (see Appendix): hence, we deduce that, some $\lambda \in \mathbb{R}$,
$$\int_{\mathbb{S}^1} \left(\kappa N \cdot \psi'' - \frac{3}{2} \kappa^2 \gamma' \cdot \psi' + (1 + \lambda)\kappa N \cdot \psi\right)\,ds \geq 0$$
provided $\psi \cdot e_3 \geq 0$ where $\gamma \cdot e_3 = \epsilon$.

Replacing $\psi$ with $\varphi - (\varphi \cdot \gamma)\gamma$ we get
\begin{equation}\label{FullEL}
\int_{\mathbb{S}^1} \left(\kappa N \cdot \varphi'' - \frac{3}{2} \kappa^2 \gamma' \cdot \varphi' + (1 + \lambda)\kappa N \cdot \varphi + \frac{1}{2} \kappa^2 \gamma \cdot \varphi \right)\,ds \geq 0
\end{equation}
provided $\psi \cdot e_3 \geq 0$ where $\gamma$ touches the obstacle. Furthermore, equality holds in (\ref{FullEL}) for variations supported in $\{\gamma \cdot e_3 > \epsilon\}$.
We prove Lemma \ref{EulerLagrange} using this form of the Euler-Lagrange equation.

\begin{proof}[{\bf Proof of Lemma \ref{EulerLagrange}}]
First, it is easy to show the energy bound using a competitor of the form $\hat \gamma=(1-\epsilon^2\eta(\theta)^2)^{1/2}(\cos(\theta) e_1 + \sin(\theta )e_2) + \epsilon \eta(\theta) e_3$, where $\eta = 1$ off of an interval of length $1/10$, and
grows to a universal height $C$ chosen so that the length constraint $\int |\hat \gamma'|=2\pi$ is satisfied.

As a consequence of the energy bound and the embedding $W^{2,2}(\mathbb S^1)\hookrightarrow L^\infty(\mathbb S^1)$, we have that $\gamma$ is at distance at most $C\epsilon$ in $L^{\infty}$ from a great circle inside $\mathbb S^2$. Since $\gamma$ is contained inside the upper hemisphere, this implies that $\gamma \subset \{x_3 \leq C\epsilon\}$.

Next, note that since $\gamma \in W^{2,2}(\mathbb{S}^1)$ we have that $\gamma'$ is continuous, so in particular is bounded.
Hence, since $\kappa\in L^2$, it follows by 
\eqref{FullEL} that
\begin{equation}\label{eq:sup}
\sup_{|\varphi|+|\varphi'| \leq 1,\, \varphi \in C^2(\mathbb S^1), \,(\varphi - (\varphi \cdot \gamma)\gamma)\cdot e_3 \geq 0} -\int_{\mathbb S^1} (\kappa N) \cdot \varphi''\,ds \leq M
\end{equation}
for some finite constant $M$.
We now note that, given any $C^1$ vector-field $\Phi:\mathbb S^1\to \mathbb R^3$, we can define
$$
\varphi(s):=\int_0^s\Phi(t)\,dt-\frac{s}{2\pi} \int_0^{2\pi}\Phi(t)\,dt+A_\Phi\,e_3\qquad \forall\,s \in [0,2\pi],
$$
where $A_\Phi>0$ is a constant to be fixed.
With this definition, $\varphi$ is a periodic vector-field of class $C^2$.
Also, because $|\gamma\cdot e_3|\leq C\epsilon$
we see that, for $\epsilon$ small,
\begin{align*}
(\varphi - (\varphi \cdot \gamma)\gamma)\cdot e_3&\geq -\|\Phi\cdot e_3\|_{L^1} -\|\Phi\cdot \gamma\|_{L^1}|\gamma\cdot e_3|+A_\Phi - A_\Phi (\gamma\cdot e_3)^2 \\
&\geq A_\Phi(1-C\epsilon^2) - \|\Phi\|_{L^1}(1+C\epsilon)>0
\end{align*}
if we chose $A_\Phi:=2\|\Phi\|_{L^1}$. This means that $\frac{\varphi}{\|\varphi\|_\infty+\|\Phi\|_\infty}$ is admissible in \eqref{eq:sup},
and we get
$$
-\int_{\mathbb S^1} (\kappa N) \cdot \Phi'\,ds \leq \Bigl(\|\varphi\|_\infty+\|\Phi\|_\infty\Bigr)M.
$$
Since $\|\varphi\|_\infty \leq C_0\|\Phi\|_{L^1}\leq  2\pi C_0\|\Phi\|_\infty$,
and $\Phi:\mathbb S^1\to \mathbb R^3$ was arbitrary,
this proves that
$$
\sup_{|\Phi| \leq 1,\, \Phi \in C^1(\mathbb S^1)} \biggl|\int_{\mathbb S^1} (\kappa N) \cdot \Phi'\,ds \biggr|\leq (1+2\pi C_0)M,
$$
which implies that $\kappa N = \gamma'' + \gamma \in BV(\mathbb{S}^1) \subset L^{\infty}(\mathbb{S}^1)$.

We show now that $(\kappa N)'$ is in fact in $L^{\infty}$. To this aim, using that $\kappa \in L^\infty$, it follows by 
\eqref{FullEL} that
$$
\sup_{\int_{\mathbb S^1}\left(|\varphi|+|\varphi'|\right) \leq 1,\, \varphi \in C^2(\mathbb S^1), \,(\varphi - (\varphi \cdot \gamma)\gamma)\cdot e_3 \geq 0} -\int_{\mathbb S^1} (\kappa N) \cdot \varphi''\,ds \leq \overline M.
$$
Hence, by the same argument as above, we get
$$
-\int_{\mathbb S^1} (\kappa N) \cdot \Phi'\,ds \leq \Bigl(\|\varphi\|_{L^1}+\|\Phi\|_{L^1}\Bigr)\overline M,
$$
and because  $\|\varphi\|_{L^1} \leq C_0\|\Phi\|_{L^1}$ we conclude that 
$$
\sup_{\|\Phi\|_{L^1} \leq 1,\, \Phi \in C^1(\mathbb S^1)} \biggl|\int_{\mathbb S^1} (\kappa N) \cdot \Phi'\,ds \biggr|\leq (1+ C_0)\overline M.
$$
As a consequence, $|(\kappa N)'| = |\gamma''' + \gamma'|$ is a bounded function, completing the proof that $\gamma\in C^{2,1}$.
\end{proof}

%%%%%%%%%%%%%%%%%%%%%%%%%%%%%%%%%%%%%%%%%%%%%%%%%%%%%%%%%%%
\subsection{$C^{2,1}$ estimate}\label{C21Estimate}
In this section we prove the $C^{2,1}$ estimate
\eqref{eq:C21 eps} in the statement of Theorem \ref{MainTheorem2},
and the fact that the contact set $\{\gamma\cdot e_3=\epsilon\}$ is nonempty.
To simplify the notation, we remove the subscript $\epsilon$ from $\gamma_\epsilon.$ 

So, let $\gamma$ be a unit-speed minimizer in $\tilde Y_{\epsilon}$ of $F$, let $h(s) := \gamma(s) \cdot e_3$, and let $\kappa$ be the geodesic curvature.
Before beginning, we note that showing $\|h\|_{C^{2,1}(\mathbb{S}^1)} \leq C\epsilon$ for universal $C$ suffices. (Here and below, $C$ denotes a large universal constant that may change from
line to line). Indeed, let $\theta$ be the angle coordinate of $\gamma$ in cylindrical coordinates with symmetry axis in the $e_3$ direction. 
A purely geometric calculation gives
\begin{equation}
\label{eq:theta h}
\theta'(s)^2 = \frac{1}{1-h^2}\left(1 - \frac{h'^2}{1-h^2}\right).
\end{equation}
As a consequence, if $\|h\|_{C^{2,1}(\mathbb{S}^1)} \leq C\epsilon$, then for $\epsilon$ small $\gamma$ is parametrized by $\alpha(\theta) := h(s(\theta))$ as
$$\gamma =\left\{(1-\alpha(\theta)^2)^{1/2}(\cos\theta\, e_1 + \sin\theta\, e_2) + \alpha(\theta)\, e_3: \theta \in \mathbb{S}^1\right\}$$
and furthermore $\|\alpha\|_{C^{2,1}(\mathbb{S}^1)} \leq C\epsilon$. We establish this estimate for $h$ below.

Of course, in what follows, we can assume that $\epsilon$ is universally small.

\begin{proof}[{\bf Proof of $C^{2,1}$ estimate}]
We first recall that, by Lemma \ref{EulerLagrange}, $|h| \leq C\epsilon$. In addition, a short computation yields the relation
\begin{equation}\label{HeightODE}
h''(s) + h(s) = \kappa N \cdot e_3.
\end{equation}
We conclude, using the energy bound $\int \kappa^2ds_\gamma \leq C\epsilon^2$
(see Lemma \ref{EulerLagrange}), that
$$\|h\|_{C^1(\mathbb{S}^1)} \leq C\epsilon.$$

In the following steps, we show that $\|\kappa\|_{C^{0,1}(\mathbb{S}^1)} \leq C\epsilon$ using energy minimality and the ODE for $\kappa$. In this way the estimate $\|h\|_{C^{2,1}(\mathbb{S}^1)} \leq C\epsilon$  will follow from the equation (\ref{HeightODE}).

\vspace{3mm}

{\bf Step $1$:} The minimum of $\kappa$ is negative. Indeed, if not then $\gamma$ encloses a convex subset of the half-sphere, and the length of $\gamma$ is strictly less than $2\pi$ (this follows, e.g., by Crofton's formula on the sphere).

\vspace{3mm}

{\bf Step $2$:} We have
$$1 + \lambda \geq 0.$$
Indeed, suppose not, and suppose that the minimum of $\kappa$ is $-A$ (with $A>0$ by Step 1) at $s = 0$. Note that $\kappa> 0$ on the contact set $\{x_3=\epsilon\}$, so this minimum must be attained in a noncontact point. Also, $\kappa$ cannot be negative everywhere. Indeed, if $\gamma(s)$ is a point which minimizes $\gamma\cdot e_3$, at this point the curvature of $\gamma$ is at least the one of the parallel at height $\gamma(s)\cdot e_3\geq \epsilon>0$, which is positive (note that, for the moment, we did not prove yet that the contact set is nonempty; this is the content of Step 7 below).

Note that, by symmetry of the ODE \eqref{ODE}, $\kappa(s)=\kappa(-s)$. Thus there exists $s_0\leq \pi$ such that $\kappa < 0$ on $(-s_0,s_0)$, with $\kappa = 0$ at $\pm s_0$, and
$$\kappa'' \leq -\frac{\kappa^3}{2}\leq A^3$$
on this interval. Integrating this information, we deduce that
\begin{equation}\label{eq:kappa A}
\kappa(s) \leq -A + \frac12 A^3s^2\qquad \text{ on $(-s_0,s_0)$},
\end{equation}
and because $\kappa(s_0)=0$ it follows that $s_0 \geq A^{-1}$. On the other hand, \eqref{eq:kappa A} also implies that $\kappa(s)\leq -A/2$ on $(-1/A,1/A)$, so 
by energy minimality it follows that 
$$
\frac{A}{2}=\int_{-1/A}^{1/A}\biggl(\frac{A}{2}\biggr)^2ds\leq \int_{-1/A}^{1/A}\kappa(s)^2ds \leq C\epsilon^2,
$$
thus $A \leq C\epsilon^2$. Combined with the bound $\pi \geq s_0 \geq A^{-1}$, this yields the desired contradiction for $\epsilon$ small enough.

\vspace{3mm}

{\bf Step $3$:} Set $\Lambda:=\sqrt{1+\lambda}$. We have
$$\Lambda \geq \frac{1}{2\pi}.$$
Indeed, the same considerations as in Step 2 give $\kappa(s) \leq -A + \frac12 A(A^2 + \Lambda^2)s^2$ for $s \in (-s_0,s_0)$, so we conclude that $(A^2 + \Lambda^2)^{-1/2} \leq s_0 \leq \pi$, and by energy minimality
that $A^2(A^2 + \Lambda^2)^{-1/2} \leq C\epsilon^2$. This yields
$$\frac{1}{\pi} \leq (A^2 + \Lambda^2)^{1/2} = A^2(A^2 + \Lambda^2)^{-1/2} + \Lambda^2(A^2 + \Lambda^2)^{-1/2} \leq C\epsilon^2 + \Lambda,$$
and the claim follows for $\epsilon$ small.

\vspace{3mm}

{\bf Step $4$:} We have
$$\Lambda^{-1}\max_{\mathbb S^1}|\kappa| \leq C\epsilon.$$
Indeed, let $A = \max_{\mathbb S^1} |\kappa|$. Using the energy estimate as in Step 3 we have $A^2(A^2 + \Lambda^2)^{-1/2} \leq C\epsilon^2$, or equivalently
$$
\frac{(A\Lambda^{-1})^2}{\sqrt{1+(A\Lambda^{-1})^2}} \leq \frac{C\epsilon^2}{\Lambda}.
$$
Thanks to Step 3, this yields
$$
\frac{(A\Lambda^{-1})^2}{\sqrt{1+(A\Lambda^{-1})^2}} \leq 2\pi C\epsilon^2,
$$
and the claim
follows easily.

\vspace{3mm}

{\bf Step $5$:} The curve $\gamma$ separates from the obstacle on intervals $I_i$ of length $\ell_i$. Let $A_i = \max_{I_i} |\kappa|$. Then
$$\sum_i A_i^2\ell_i \leq C\epsilon^2.$$
Indeed, the ODE has the conserved quantity
\begin{equation}
\label{eq:energy cons}
\kappa'^2 + \Lambda^2\kappa^2 + \frac{\kappa^4}{4}={\rm const} \qquad \text{on each $I_i$}.
\end{equation}
This implies that, if $\kappa$ has multiple local maxima or minima inside $I_i$, then at these points the value of $|\kappa|$ is equal to $A_i$
since $\kappa'=0$ there.
Thus, if we write $I_i=\cup_{j}I_{ij}$ where $\kappa$ has constant sign inside $I_{ij}$,
and if $x_{ij}\in I_{ij}$ is a local maximum for $|\kappa|$, we have $|\kappa(x_{ij})|=A_i$.

Also, it follows by Step 4 and the ODE
$$
\kappa''=-(\Lambda^2+\kappa^2/2)\kappa=-\Lambda^2(1+\Lambda^{-2}\kappa^2/2)\kappa
$$
that $\kappa$ is concave while positive and convex while negative.
Hence, $|\kappa|$ is concave inside each interval $I_{ij}.$ This implies that its graph stays above the triangle that has basis $I_{ij}\times \{0\}$ and vertex at $(x_{ij},A_{i})$,
therefore
$$
\int_{I_{ij}}\kappa^2ds \geq \frac12 A_i^2\mathcal H^1(I_{ij}).
$$
Adding these inequalities over $i$ and $j$ we obtain (since $\sum_j\mathcal H^1(I_{ij})=\ell_i$)
$$\frac12 \sum_i A_i^2\ell_i\leq \int_{\mathbb S^1}\kappa^2ds,
$$
and we conclude by energy minimality (see Lemma \ref{EulerLagrange}).
\vspace{3mm}

{\bf Step $6$:} We have
$$\sum_{i} A_i^2 \ell_i^3 \geq \frac{\epsilon^2}{4}.$$
This follows from the constraint $\int_{\mathbb{S}^1} \theta'(s)\,ds = 2\pi$
(recall that $\theta$ is the angle coordinate of $\gamma$ in cylindrical coordinates with symmetry axis in the $e_3$ direction).  Indeed, since $\|h\|_{C^1(\mathbb{S}^1)} \leq C\epsilon$ and $h \geq \epsilon$ (recall that  $h = \gamma \cdot e_3$), we have for $\epsilon$ small that
$$\theta'(s) \geq (1 + \epsilon^2)^{1/2}\left(1 - \frac{3}{2}h'^2\right)^{1/2} \geq 1 + \frac{\epsilon^2}{4} - h'^2.$$
Integrating we conclude that
\begin{equation}\label{LengthDifference}
\int_{\mathbb{S}^1} h'^2\,ds \geq \frac{\epsilon^2}{4}.
\end{equation}
Suppose that $0\in \overline{I_i}$ is a minimum point for $h$, so that $h(0) \geq \epsilon$ and $h'(0) = 0$. Multiply \eqref{HeightODE} by $h'$ and integrate on $I_i$ to obtain 
$$h'^2(s) \leq \int_{0}^s A_i|h'|\,ds \leq A_i\ell_i^{1/2} \left(\int_0^{\ell_i} h'^2\,ds\right)^{1/2}.$$
Integrating again on $I_i$, we obtain
$$\int_{I_i} h'^2\,ds \leq A_i^2\ell_i^3.$$
Since $h' = 0$ on $\{h = \epsilon\}$, the claim follows from the inequality (\ref{LengthDifference}).

\vspace{3mm}

{\bf Step $7$:}
The contact set $\{\gamma\cdot e_3=\epsilon\}$ is nonempty.

Assume by contradiction that this is not the case. This implies that the ODE \eqref{ODE} holds with equality on the whole $\mathbb S^1$. Also, by Step 5, because $\ell_1=2\pi$ (there is only one interval where $\gamma$ detaches from the obstacle) we get that $|\kappa|\leq C\epsilon$.
Hence, integrating \eqref{ODE} on $\mathbb S^1$ and using Step 4,
we get (since $\int_{\mathbb S^1} \kappa''ds=0$ by periodicity)
$$
\biggl|\int_{\mathbb S^1} \kappa\,ds\biggr| =\frac{\Lambda^{-2}}{2}\int_{\mathbb S^1} \kappa^3 \leq C\epsilon^3.
$$
Then, integrating \eqref{HeightODE} and using that $|N\cdot e_3-1|\leq C\epsilon^2$ (since $|\gamma\cdot e_3|+|\gamma'\cdot e_3|\leq C\epsilon$) we get
$$
\int_{\mathbb S^1}h\,ds=\int_{\mathbb S^1}\kappa N\cdot e_3\,ds \leq \int_{\mathbb S^1}\kappa\,ds+C\epsilon^2\int_{\mathbb{S}^1}|\kappa|\,ds
\leq C\epsilon^3,
$$
where we used again that $|\kappa|\leq C\epsilon$.
However, this is a contradiction since $h=\gamma\cdot e_3 \geq \epsilon$ everywhere.

\vspace{3mm}

{\bf Step $8$:} We have
$$\Lambda \leq C.$$
Indeed, suppose by way of contradiction that $\Lambda^2 >>1$. The inequalities from Step $5$ and Step $6$ imply that some $\ell_i$ (say $\ell_1$) is larger than some universal constant $c_1 > 0$.
So, by Step 5 again, $A_1^2 \leq c_1^{-1}C\epsilon^2$, thus $|\kappa| < C\epsilon$
on $I_1$. The idea of the following argument is that if $\Lambda$ is too large, then $\kappa$ oscillates rapidly around $0$, so $\gamma$ follows a great
circle that is tangent to the obstacle from below, contradicting that $h \geq \epsilon$.

We now establish this rigorously. By Step 7, $I_1$ cannot coincide with the whole $\mathbb S^1$. Assume that $I_1$ starts at $0$. Then for some $s_0$, on $I_1$ we claim that
$$\kappa = C_0\epsilon \sin(\Lambda (s- s_0)) + O(\epsilon^3)$$
for universal $C_0$. (Here and below $O(\delta)$ indicates a function that is smaller in absolute value than $C \delta$ for universal $C$).

Indeed, since $|\kappa| \leq C\epsilon$ on $I_1$ we have $\kappa'' + \Lambda^2\kappa = O(\epsilon^3)$. Take $C_0$ and $s_0$ such that $C_0\epsilon \sin(\Lambda (s - s_0))$ has the same initial value and derivative 
as $\kappa$ at $0$.

We first claim that $C_0 \leq C$ universal.
Indeed, we note that 
$$
\kappa'(0)^2 + \Lambda^2\kappa(0)^2 = C_0^2\Lambda^2\epsilon^2\cos^2(\Lambda s_0)+C_0^2\Lambda^2\epsilon^2\sin^2(\Lambda s_0)=C_0^2\epsilon^2\Lambda^2.
$$
Hence, by the conservation law
$\kappa'^2 + \Lambda^2\kappa^2 + \kappa^4/4 = {\rm const}$,
if $\bar s \in I_1$ denotes a maximum point of $|\kappa|$ then (note that $\kappa'(\bar s)=0$)
$$
C_0^2\epsilon^2\Lambda^2=
\kappa'(0)^2+\Lambda^2\kappa(0)^2\leq 
\Lambda^2\kappa(\bar s)^2 + \frac{\kappa(\bar s)^4}4\leq C\Lambda^2\epsilon^2+C\epsilon^4,
$$
and the claim follows (recall that, by assumption, $\Lambda$ is large).

Now, consider the function $w(s) := \kappa(s) - C_0\epsilon\sin(\Lambda(s-s_0))$. Then $w$ solves
$$w'' + \Lambda^2 w = O(\epsilon^3),\quad w(0) = w'(0) = 0.$$
Multiplying by $w'$ and integrating we obtain that
$$w'^2(t) + \Lambda^2 w^2(t) < C\epsilon^3\int_0^t |w'|\,ds < Ct\epsilon^3 \max_{I_1}|w'|.$$
Choosing first $t \in I_1$ where $|w'|$ attains its maximum we obtain $\max_{I_1} |w'| \leq C\epsilon^3$. Then, combining this information with
the above inequality we get
$\max_{I_1} |w|  \leq C\Lambda^{-1}\epsilon^3 \leq C\epsilon^3$, that is
$$
\max_{I_1}|\kappa(s)-C_0\epsilon\sin(\Lambda(s-s_0))|\leq C\epsilon^3.
$$
Since $N \cdot e_3 = 1 + O(\epsilon^2)$ (by the inequality $\|h\|_{C^1(\mathbb{S}^1)} \leq C\epsilon$), it follows by \eqref{HeightODE} that
$$h'' + h = C_0\epsilon \sin(\Lambda (s-s_0)) + O(\epsilon^3).$$ 
Using the initial conditions $h(0) = \epsilon$ and $h'(0) = 0$ (since $0$ is a contact point) we obtain, in a similar way to above, that 
\begin{align*}
h(s) &= \frac{C_0 \epsilon}{1-\Lambda^2}\sin(\Lambda(s-s_0)) + \epsilon\left(1 + \frac{C_0}{1-\Lambda^2}\sin(\Lambda s_0)\right)\cos(s) \\
& - \frac{C_0 \epsilon\Lambda}{1-\Lambda^2}\cos(\Lambda s_0)\sin(s) + O(\epsilon^3)=\epsilon \cos(s)+O\biggl(\frac{\epsilon}{\Lambda}+\epsilon^3\biggr).
\end{align*}
Taking $s = \min\{\pi/2,c_1\}$ we get a contradiction to $h \geq \epsilon$ for $\Lambda$ sufficiently large.

\vspace{3mm}

{\bf Step $9$:} By Steps $4$ and $8$ we have
$$\max|\kappa| \leq C\epsilon.$$
Since the ODE for $\kappa$ on $I_i$ has the conserved quantity $\kappa'^2 + \Lambda^2 \kappa^2 + \kappa^4/4$, using that $\kappa'(s_i)=0$ if $\max_{I_i} |\kappa| = |\kappa(s_i)|$
we conclude that
$$
\max_{I_i}|\kappa'|^2 \leq \Lambda^2 \kappa(s_i)^2 + \frac{\kappa(s_i)^4}4 \leq C\epsilon^2\qquad \forall\,i.
$$
Since $\kappa'=0$ on the contact set, we conclude that
$|\kappa'| \leq C\epsilon$.
Recalling \eqref{HeightODE}, this proves that
$\|h\|_{C^{2,1}(\mathbb S^1)}\leq C\epsilon$, completing the proof.
\end{proof}

As a consequence of the $C^{2,1}$ estimate, we can show that for $\epsilon$ small, the minimizer $\gamma$ cannot lift from the obstacle on short intervals. In the following we take $\epsilon < \epsilon_0$ small
universal.

\begin{lem}\label{ShortBumps}
There exists $c_0 > 0$ universal such that if one of the intervals $I$ in $\{h > \epsilon\}$ satisfies $\mathcal{H}^1(I) < c_0$, then
$\kappa > \epsilon(1-\epsilon^2)^{-1/2}$ on $I$.
\end{lem}
\begin{proof}
A short computation shows that the curvature of the obstacle $\mathbb{S}^2 \cap \{x_3 = \epsilon\}$ is $\epsilon(1-\epsilon^2)^{-1/2}$. Since the obstacle touches $\gamma$ from below on contact points,
we have $\kappa \geq \epsilon(1-\epsilon^2)^{-1/2}$ at the endpoints of $I$. 

Assume that $\kappa \leq \epsilon(1-\epsilon^2)^{-1/2}$
somewhere in $I$.
Since $\kappa$ is strictly concave where it is positive (see Step 5 above in the proof of the $C^{2,1}$ estimates), it has to become negative somewhere inside $I$ (otherwise it could not reach the value $\epsilon(1-\epsilon^2)^{-1/2}$ at the boundary points of $I$), implying that
$$
{\rm osc}_I \kappa \geq \epsilon(1-\epsilon^2)^{-1/2}.
$$
However, by the mean value theorem
$$
{\rm osc}_I \kappa \leq \mathcal H^1(I)\max_I|\kappa'|.
$$
Since $|\kappa'| \leq C\epsilon$ (see Step 9 above),
the two inequalities above imply
$$
\mathcal H^1(I)\geq C^{-1},
$$
as desired.
\end{proof}

\begin{lem}\label{TouchingInequality}
If an interval $I$ in $\{\alpha(\theta) > \epsilon\}$ satisfies $\mathcal{H}^1(I) < \pi,$ then $\kappa \leq \epsilon(1-\epsilon^2)^{-1/2}$ somewhere in $I$.
\end{lem}
\begin{proof}
The proof is by elementary geometry on $\mathbb{S}^2$ and by maximum principle.

Suppose that $I$ is centered at $\theta = 0$. Let $\xi_{\varphi} \in \mathbb{S}^2$ parametrize the half great-circle of vectors in $\mathbb{S}^2$ with 
$\theta$-coordinate equal to $\pi$ and angle $\varphi \in (0,\, \pi)$ from the $e_3$ axis.
Note that the circles $K_{\varphi} := \{\xi_{\varphi} \cdot x = \epsilon\} \cap \mathbb{S}^2$ intersect the obstacle $K_0$ at points
with $\theta$-coordinate in $(\pi/2,\, 3\pi/2)$. It follows that, for some $\varphi_0 \in (0,\, \pi/2)$, the circle $K_{\varphi_0}$ touches $\gamma$ from above locally at a point with $\theta$ coordinate in $I$
(see Figure \ref{ShortBumps_Pic}). Then,
since the curvature of $K_{\varphi_0}$ is $\epsilon(1-\epsilon^2)^{-1/2}$, the desired inequality holds at this contact point.
\end{proof}

 \begin{figure}
 \centering
    \includegraphics[scale=0.25]{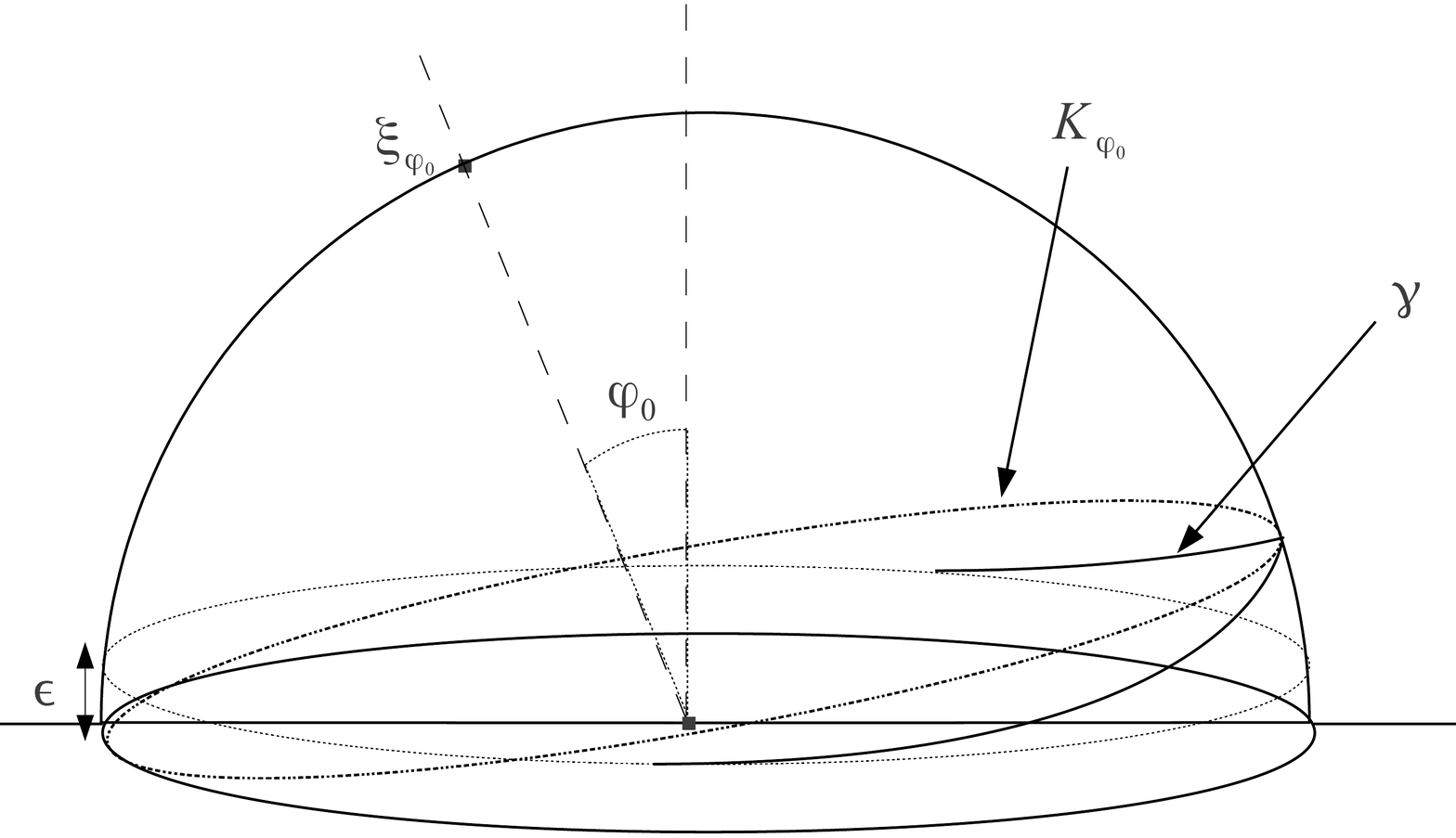}
 \caption{}
 \label{ShortBumps_Pic}
\end{figure}

As a consequence of Lemmas \ref{ShortBumps} and \ref{TouchingInequality} we have:
\begin{prop}\label{BumpLength}
There exists a universal $c_0 > 0$ such that for all $\epsilon < \epsilon_0$ universal, the intervals $I_i$ that comprise $\{h > \epsilon\}$ satisfy $\mathcal{H}^1(I_i) \geq c_0$.
\end{prop}
\begin{proof}
Assume by way of contradiction that $\mathcal H^1(I_1) < c_0$ small. By Lemma \ref{ShortBumps} we have that $\kappa > \epsilon(1-\epsilon^2)^{-1/2}$ on $I_1$. Since $\|h\|_{C^{1}(\mathbb{S}^1)} < C\epsilon$, the angle 
in the $\theta$ variable subtended by this interval is at most
$2c_0 < \pi$ for $\epsilon$ small. This is a contradiction of Lemma \ref{TouchingInequality}.
\end{proof}

In particular, for $\epsilon < \epsilon_0$ small universal, the set $\{h > \epsilon\}$ consists of finitely many intervals of length at least $c_0$.

\vspace{3mm}

%%%%%%%%%%%%%%%%%%%%%%%%%%%%%%%%%%%%%%%%%%%%%%%%%%%%%%%%%%%%%%%%%%%%%%%%%%%%%

\subsection{Linear Problem}
Let $\gamma_{\epsilon}$ be a minimizer in $Y_{\epsilon}$ of $F$, with geodesic curvature $\kappa_{\epsilon}$. Let $h_{\epsilon}(s) = \gamma_{\epsilon}(s) \cdot e_3$, and let $\Lambda_{\epsilon}^2 = 1 + \lambda_{\epsilon}$ be
the Lagrange multiplier.

We consider the problem obtained by ``stretching the picture vertically.'' So, we set 
$$\tilde{h}_{\epsilon} := \epsilon^{-1}h_{\epsilon}\quad \text{and} \quad \tilde{\kappa}_{\epsilon} := \epsilon^{-1}\kappa_{\epsilon}.$$ 
By the $C^{2,1}$ estimate proved in the previous section, there exists a universal constant $C$ such that
$$\tilde{h}_{\epsilon} \geq 1,\, \|\tilde{h}_{\epsilon}\|_{C^{2,1}(\mathbb{S}^1)} \leq C$$
$$\tilde{h}_{\epsilon}'' + \tilde{h}_{\epsilon} = \tilde{\kappa}_{\epsilon} + O(\epsilon^2),$$
$$\tilde{\kappa}_{\epsilon}'' + (\Lambda_{\epsilon}^2 + \epsilon^2 \tilde{\kappa}_{\epsilon}^2 / 2)\tilde{\kappa}_{\epsilon} \geq 0\quad \text{with equality where } \tilde{h}_{\epsilon} > 1,$$ 
$$\Lambda_{\epsilon}^2 \in (C^{-1},\,C).$$
Moreover, recalling the identity \eqref{eq:theta h}, the length constraint $\int_{0}^{2\pi} \theta'(s)\,ds = 2\pi$ reads
$$\int_{\mathbb{S}^1} (\tilde{h}_{\epsilon}^2 -\tilde{h}_{\epsilon}'^2)\,ds = O(\epsilon^2).$$

In the limit that $\epsilon \rightarrow 0$, the functions $\tilde{h}_{\epsilon}$ and $\tilde{\kappa}_{\epsilon}$ converge (up to taking a subsequence) in $C^2$ (resp. $C^0$) to a solution of the following linear problem:
\begin{equation}\label{PreciseLinearProblem}
\text{Linear Problem} = \begin{cases}
h \geq 1, \, h \in C^{2,1}(\mathbb{S}^1) \\
h'' + h = \kappa \\
\kappa'' + \Lambda^2 \kappa \geq 0 \text{ and } = 0 \text{ where } h > 1 \\
\int_{\mathbb{S}^1} (h'^2 - h^2)\,ds = 0.
\end{cases}
\end{equation}

We now describe precisely the minimizers of $\int_{\mathbb{S}^1} \kappa^2\,ds$ over all $h,\,\kappa$ satisfying the linearized problem.
Note that this result was already numerically predicted in \cite{CM}.

\begin{prop}\label{LinearMinimizers}
Let $h,\, \kappa$ be a minimizer of $\int_{\mathbb{S}^1} \kappa^2\,ds$ over all pairs that solve the linear problem (\ref{PreciseLinearProblem}). Then $h = 1$ on $\mathbb{S}^1 \backslash I$, where $I$ is an open interval of length between $2.42$ and $2.43$. Furthermore, 
$\Lambda \in (3.79,3.82)$, and $h$ and $\kappa$ are given by explicit formulae (in particular, they are unique), see Figures \ref{h_pic} and \ref{curvature_pic}.
\end{prop}

 \begin{figure}
 \centering
    \includegraphics[scale=0.25]{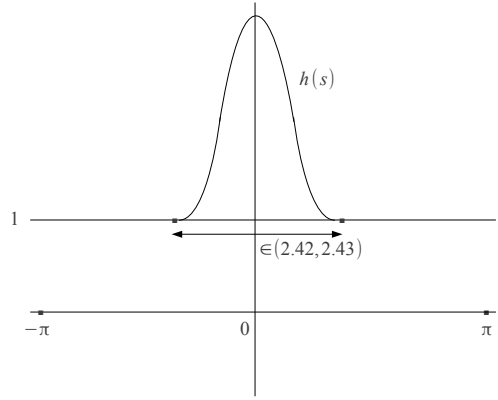}
 \caption{For a minimizer, $h$ lifts from the obstacle on exactly one interval.}
 \label{h_pic}
\end{figure}

 \begin{figure}
 \centering
    \includegraphics[scale=0.25]{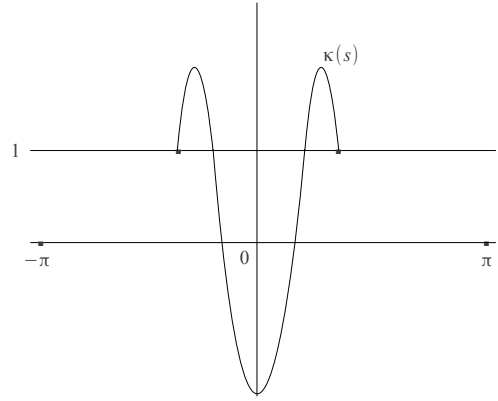}
 \caption{The function $\kappa$ is Lipschitz and solves the ODE $\kappa'' + \Lambda^2 \kappa = 0$ on the interval $\{h > 1\}$.}
 \label{curvature_pic}
\end{figure}

\begin{proof}
First, we claim that the contact set $\{h = 1\}$ is nontrivial. If not, integrating the equations $h''+h=\kappa$ and $\kappa''+\Lambda^2\kappa=0$, it follows by periodicity that
$$\int_{\mathbb{S}^1} h\,ds = \int_{\mathbb{S}^1} \kappa\,ds =-\Lambda^{-2}\int_{\mathbb S^1}\kappa''ds= 0,
$$ contradicting that $h \geq 1$.

\vspace{3mm}

On an interval $(-s_0,\,s_0) \subset \{h > 1\}$ with $h = 1$ at the endpoints, one explicitly solves the equations for $\kappa$ and $h$ to obtain
\begin{equation}
\label{eq:formulas k h}
\kappa = \frac{\cos(\Lambda s)}{\cos(\Lambda s_0)}, \quad h(s) = \frac{\sin(s_0)\cos(\Lambda s) - \Lambda \sin(\Lambda s_0)\cos(s)}{\sin(s_0)\cos(\Lambda s_0) - \Lambda \sin(\Lambda s_0)\cos(s_0)}.
\end{equation}
By the $C^{2,1}$ regularity for $h$ and the ODE $h'' + h = \kappa$ we get the relation
\begin{equation}
\label{eq:relation s L}
\frac{\tan(\Lambda s_0)}{\Lambda s_0} = \frac{\tan(s_0)}{s_0}.
\end{equation}
Since $s_0 < \pi$ and $\tan(z)/z$ is injective on $(0,\,\pi)$, it follows that $\Lambda \geq 1$. Furthermore, the ODE has no nontrivial solutions when $\Lambda = 1$. We conclude that
$$\Lambda > 1.$$

The set $\{h > 1\}$ consists of open intervals $I_i$ of length $2s_i$. Using the above computations, we rewrite the constraint $\int_{\mathbb{S}^1} (h^2 - h'^2)\,ds = 0$. Using that $h = 1$ and $h' = 0$
on $\{h = 1\}$, as well as the explicit formula for $h$ given in \eqref{eq:formulas k h}, we see that the constraint equation is equivalent to
$$ 2\pi+\sum_i(\Delta_i - 2s_i)=0,$$
where
$$ \Delta_i := \frac{s_i^3\sec^2(\Lambda s_i)-\tan(s_i)[2\Lambda^2 s_i^2-s_i^2]}{(1-\Lambda^2)s_i^2} $$
and (recall \eqref{eq:relation s L})
\begin{equation}
\label{eq:relation s L i}
\frac{\tan(s_i)}{s_i}=\frac{\tan(\Lambda s_i)}{\Lambda s_i}.
\end{equation}
Since
$$\sec^2(\Lambda s_i)=1+\tan^2(\Lambda s_i)=1+\frac{\tan^2(\Lambda s_i)}{\Lambda^2 s_i^2}\Lambda^2 s_i^2=1+\frac{\tan^2(s_i)}{s_i^2}\Lambda^2 s_i^2=1+\Lambda^2 \tan^2(s_i),$$
we can rewrite
\begin{align*}
\Delta_i-2s_i&=\frac{s_i^3[1+\Lambda^2 \tan^2(s_i)]-\tan(s_i)[2\Lambda^2s_i^2-s_i^2]-2s_i^3+2\Lambda^2s_i^3}{(1-\Lambda^2)s_i^2}\\
&=-\frac{\Lambda^2 s_i\tan^2(s_i)-(2\Lambda^2-1)[\tan(s_i)-s_i]}{\Lambda^2-1}.
\end{align*}
Thus, the constraint can be rewritten as
\begin{equation}\label{LinearConstraint}
\Lambda^2=\frac{2\pi +\sum_i[\tan(s_i)-s_i]}{2\pi -\sum_i[s_i\tan^2(s_i)-2(\tan(s_i)-s_i)] }.
\end{equation}

\vspace{3mm}

We now claim that $s_i <\pi/2$ for all $i.$ Indeed, if not, assume that $s_1\geq \pi/2$
and note that $s_j < \pi/2$ for $j \geq 2$
(because, being the intervals disjoint and the contact set nonempty, $\sum_i s_i< \pi$). Consider the function 
$$
g(z) := z\tan^2(z) - 2(\tan(z) - z),
$$
so that the denominator in (\ref{LinearConstraint}) is given by $2\pi - \sum_{i} g(s_i)$. 

It is easy to check that
$g \geq 0$ on $(0,\,\pi)$, and that $g$ is decreasing on $(\pi/2,\,\pi]$ with $g(\pi) = 2\pi$. In particular, $g(s_1) > 2\pi$.
Thus we have
\begin{align*}
\Lambda^2 &= \frac{s_1 - \tan(s_1) -2\pi -\sum_{i\geq 2}[\tan(s_i)-s_i]}{g(s_1)-2\pi +\sum_{i\geq 2} g(s_i) } \\
&\leq \frac{s_1-\tan(s_1)-2\pi }{g(s_1)-2\pi } \\
&\leq 1, 
\end{align*}
where we used that $\tan(s_i)>s_i$ for $i > 2$, $g(s_i)\geq 0$,
and that the function
$$s \mapsto \frac{s-\tan(s)-2\pi }{g(s) -2\pi }$$
is bounded by $1$ on $[\pi/2,\pi)$. This contradicts that $\Lambda > 1$, proving the claim.

\vspace{3mm}

Using (\ref{LinearConstraint}) again, we can now improve the bound on $s_i$. Since all $s_i$ are less than $\pi/2$, the numerator in the expression for $\Lambda^2$ is positive, therefore so is the denominator. In particular, this implies that $\sum_ig(s_i) < 2\pi$.
Since $g$ is increasing on $(0,\, \pi/2)$
and $g(1.225)>2\pi$, we have
\begin{equation}
\label{eq:s i bdd}
g(s_i) \leq 2\pi \quad \Rightarrow \quad s_i \leq s_c,\qquad \forall\,i,
\end{equation}
where $s_c \in (0, 1.225)$ is the unique point such that $g(s_c)=2\pi$.

%------

\vspace{3mm}

The computations so far only used that $h$ and $\kappa$ solve the linear problem. We now bring in the energy minimality. Using
the formula for $\kappa$ in \eqref{eq:formulas k h} we obtain
$$\int_{\mathbb S^1}\kappa^2ds = 2\pi+\sum_i (E_i-2s_i),$$
where
$$E_i:=\tan(s_i)+s_i[1+\Lambda^2 \tan^2(s_i)].$$
The minimization problem can thus be rewritten as
$$\min \biggl\{\sum_i \bigl[(\tan(s_i)-s_i)+\Lambda^2 s_i\tan^2(s_i)\bigr]:\text{\eqref{eq:relation s L i} and \eqref{LinearConstraint} hold}\biggr\}.$$
Using the constraint (\ref{LinearConstraint}) to rewrite the term $s_i\tan^2(s_i)$, we get that the problem is equivalent to minimizing the energy
\begin{equation}\label{EnergyFormula}
\mathcal E:=\Lambda^2\bigl(\pi+\sum_i(\tan(s_i)-s_i)\bigr)
\end{equation}
under \eqref{eq:relation s L i} and \eqref{LinearConstraint}.
\vspace{3mm}

We now want to analyze better the constraint \eqref{eq:relation s L i}. To this aim, we note that
the relation
$$\frac{\tan(s)}{s}=\frac{\tan(\Lambda s)}{\Lambda s},\qquad \Lambda > 1,$$
gives, for any $s \in (0,\pi)$, a sequence of solutions
$$1 < \Lambda_{1,s} < \Lambda_{2,s} < \ldots$$
These are found by imposing
$$\tan(s\,\Lambda_{j,s})-\Lambda_{j,s}\tan(s)=0, \qquad s\,\Lambda_{j,s} \in (j\pi,(j+1)\pi),$$
(see Figure \ref{C21Constraint}).
 \begin{figure}
 \centering
    \includegraphics[scale=0.25]{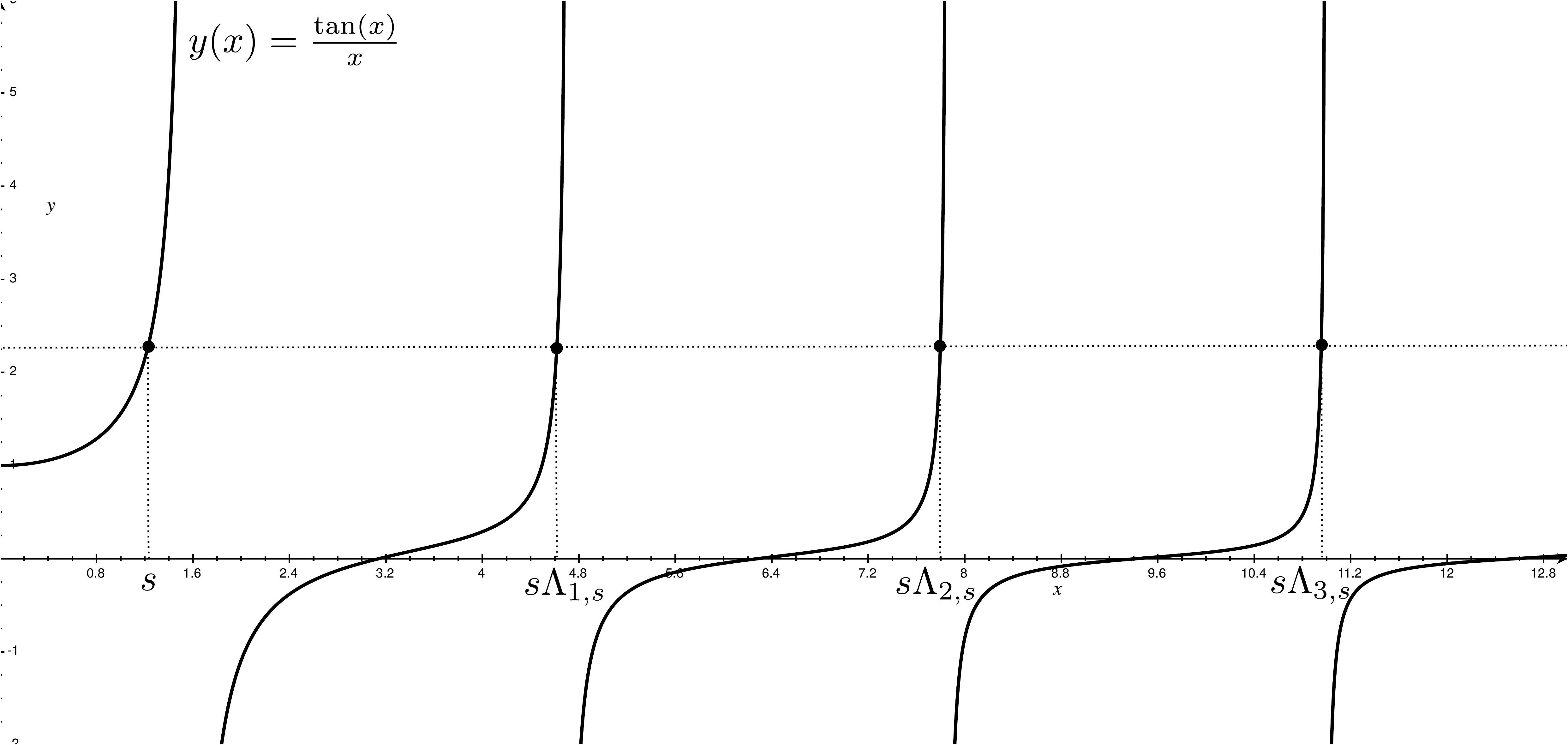}
 \caption{For $y(x) = \frac{\tan(x)}{x}$, the functions $\Lambda_{j,s}$ are found by imposing $y(s\Lambda_{j,s}) = y(s)$.}
 \label{C21Constraint}
\end{figure}
By implicitly differentiating, we see that the functions $s \mapsto \Lambda_{j,s}$ are strictly decreasing on $(0,\pi)$. Furthermore, $\Lambda_{j,s}\to \infty$ as $s\to 0$.

\vspace{3mm}

Using these observations we estimate the minimal energy from above. To do so we consider the case that $\{h > 1\}$ consists of one interval.

We first estimate the length of this interval. Recall that, by \eqref{eq:s i bdd}, we can restrict to the range $s \in (0,s_c)$. Because
$$s\mapsto \mathcal L(s):=\frac{2\pi +[\tan(s)-s]}{2\pi -[s\tan^2(s)-2(\tan(s)-s)] } \quad \text{is increasing on $(0,s_c)$},$$
$$s\mapsto \Lambda_{1,s} \quad \text{is decreasing on $(0,s_c)$},$$
$$\text{$\Lambda_{1,s}\to +\infty$ as $s\to 0^+$,}\qquad \text{$\mathcal L(s)\to +\infty$ as $s\to s_c^-$,}$$
there exists a unique point $\hat s_1<s_c < 1.225$ such that
$$\mathcal L(\hat s_1)=\Lambda^2_{1,\hat s_1}.$$
Note that for $x > 0$ the function $\frac{\tan{x}}{x}$ is increasing on each interval $((k-1/2)\pi,\, (k + 1/2)\pi)$. Since
$$ \frac{\tan(3.81*1.21)}{3.81 * 1.21} - \frac{\tan(1.21)}{1.21} <0,
\qquad \frac{\tan(3.82*1.21)}{3.82 * 1.21} - \frac{\tan(1.21)}{1.21} >0,$$
we conclude that $3.81\leq \Lambda_{1,1.21}\leq 3.82$. Also, since
$\mathcal L(1.21)<(3.81)^2$, we have
 $\mathcal L(1.21)<\Lambda_{1,1.21}^2$
which yields $\hat s_1 \geq 1.21$ and $\Lambda_{\hat s_1}\leq \Lambda_{1,1.21}\leq 3.82$.

On the other hand
$$ \frac{\tan(3.79*1.215)}{3.79 * 1.215} - \frac{\tan(1.215)}{1.215} <0,
\qquad \frac{\tan(3.8*1.215)}{3.8 * 1.215} - \frac{\tan(1.215)}{1.215} >0,
$$
together with $\mathcal L(1.215)>(3.8)^2$,
implies that $\hat s_1 \leq 1.215$ and $\Lambda_{\hat s_1}\geq 3.79$.
Thus we proved that
\begin{equation}
\label{eq:hat s1}
\hat s_1 \in (1.21,1.215)\quad \text{and}\quad \Lambda_{\hat s_1}\in (3.79,3.82)
\end{equation}

These estimates give in particular an upper bound for the minimal energy:
$$\min \mathcal E \leq \Lambda_{\hat s_1}^2[\pi +\tan(\hat s_1) - \hat s_1]\leq  (3.82)^2[\pi + \tan(1.215) - 1.215]\leq 67.4.$$

\vspace{3mm}

Using the energy bound we now get an upper bound for $\Lambda$ for a minimizer. Indeed, the expression (\ref{EnergyFormula}) gives (note that $\tan(s_i)\geq s_i$ since $s_i <\pi/2$)
\begin{equation}
\label{eq:upper lambda}
\Lambda^2 \pi < \min \mathcal E \leq 67.4 \quad  \Longrightarrow \quad \Lambda \leq 4.64.
\end{equation}
Recall now that $s_i \leq 1.225$ for all $i$.
Since 
$$
\frac{\tan(3.75*1.225)}{3.75 * 1.225} - \frac{\tan(1.225)}{1.225} <0,\qquad
\frac{\tan(6.35*1.225)}{6.35 * 1.225} - \frac{\tan(1.225)}{1.225} <0,
$$
$$
3.75*1.225\in (\pi,2\pi),\qquad 6.35*1.225 \in (2\pi,3\pi),
$$
we conclude that, for all $i$, 
\begin{equation}
\label{eq:lower lambda}
\Lambda_{1,s_i}\geq \Lambda_{1,1.225}\geq 3.75,
\end{equation}
$$\Lambda_{2,s_i}\geq \Lambda_{2,1.225}\geq 6.35.$$
Since $4.64 < \Lambda_{2,1.225}$ (see \eqref{eq:upper lambda}), we deduce that $\Lambda=\Lambda_{1,s_i}$ for all $i$.
Since $\Lambda_{1,s}$ is strictly decreasing,
this implies that there are a finite number $N$
of folds with identical length: $s_1=s_2=\ldots=s_N=\bar s$. Thus (\ref{LinearConstraint}) reads
$$\Lambda^2=\frac{2\pi +N[\tan(\bar s)-\bar s]}{2\pi -N[\bar s\tan^2(\bar s)-2(\tan(\bar s)-\bar s)] }.$$
Since $N\bar s< \pi$, this gives
$$\Lambda^2 \leq \frac{1+\tan(\bar s)/\bar s}{2\tan(\bar s)/\bar s-\tan(\bar s)^2} \leq \frac{2}{2 - \bar s\tan(\bar s)}$$
that combined with the lower bound $\Lambda \geq 3.75$ (see \eqref{eq:lower lambda}) yields $\bar s >\pi/3$, so this immediately gives $N=1$ or $N=2$.

In the computation above we have shown that, if we consider one single fold, then we can make the energy lower than $67.4$.
We now want to prove that $N=2$ is energetically less efficient.
\vspace{3mm}

Observe that $\frac{\tan(x)}x > 1$ on $(0,\, \pi/2)$, and is increasing on $(\pi/2,\, 3\pi/2).$ Since $\tan(1.43\pi)/(1.43\pi) < 1$, we conclude that 
$$\Lambda s > 1.43 \pi.$$
Because of this, in the case $N = 2$ we have that the energy is at least
$$ \min_{s \in (\pi/3,1.225)}\left(\frac{1.43 \pi}{s}\right)^2[\pi + 2(\tan(s) - s)].
$$
To provide a lower bound on the above quantity, we check that at the end points it is larger than $80$.
Also, if there is a critical point $s_0 \in (\pi/3,1.225)$, then at such a point we have (since the first derivative vanishes) $\pi + 2(\tan(s_0) - s_0) = s_0\tan^2(s_0)$, so the energy at $s_0$ is $(1.43\pi)^2 (\tan^2(s_0)/s_0)$. The critical
point happens for $s > 1.13$ by a simple computation (the right side in the critical point condition has larger derivative than left side, and the difference changes sign between $1.13$ and $1.14$). 
Thus, since $\frac{\tan^2(s)}{s}$ is increasing, we deduce that the energy at a critical point is at least $(1.43\pi)^2(\tan(1.13))^2/(1.13) > 80.$

Since $80 > 67.4$, this shows that the case $N=2$ has higher energy than $N=1$. We conclude that, for a minimizer, $\{h > 1\}$ consists of exactly one interval of length $2\hat s_1$. Furthermore, thanks to \eqref{eq:hat s1},
$$2.42 \leq 2\hat s_1 \leq 2.43,$$
completing the proof.
\end{proof}

%%%%%%%%%%%%%%%%%%%%%%%%%%%%%%%%%%%%%%%%%%%%%%%%%%%%%%%%%%
\subsection{Proof of Theorem \ref{MainTheorem2}}
By combining the $C^{2,1}$ estimate with Proposition \ref{LinearMinimizers}, we prove Theorem \ref{MainTheorem2}.

\begin{proof}[{\bf Proof of Theorem \ref{MainTheorem2}}]
We proved the $C^{2,1}$ estimate in Section \ref{C21Estimate}.
Recall that, as a result of this estimate, as $\epsilon \rightarrow 0$, there is a subsequence of $\epsilon_k^{-1}h_{\epsilon_k},\, \epsilon_k^{-1}\kappa_{\epsilon_k}$ (corresponding
to minimizers $\gamma_{\epsilon_k}$ of $F$ in $ Y_{\epsilon_k}$) that converge respectively in $C^2$ and $C^0$ to a solution $h,\, \kappa$ of the linear problem (\ref{PreciseLinearProblem}).

\vspace{3mm}
We first claim that this limit is a minimizer of $\int_{\mathbb{S}^1} \kappa^2\,ds$.
By strong $C^2$ convergence it is clear that $\lim_{k \rightarrow \infty} \epsilon_k^{-2}F(\gamma_{\epsilon_k})$ is at least the minimal energy for the linear problem. If $\lim\sup_{k \rightarrow \infty} \epsilon_k^{-2}F(\gamma_{\epsilon_k})$ is 
larger than the minimal energy for the linear problem,
then by using the linear minimizer and making arbitrarily small perturbations to satisfy the length constraint, we get a competitor of $\gamma_{\epsilon_k}$ with smaller energy, which would give a contradiction.
(More precisely, if $h$ is the minimizer of the linearized problem, create a competitor by perturbing the curve
$(1-\epsilon_k^2h(\theta)^2)^{1/2}(\cos(\theta)\, e_1 + \sin(\theta)\, e_2) + \epsilon_k h(\theta) e_3.$ A short computation shows that to satisfy the length constraint we need to make a perturbation to $h$ of size $\epsilon_k^2$
in $C^2$. The curvature of the competitor is then $\epsilon_k \kappa + O(\epsilon_k^2),$ so for $k$ large the energy of the competitor is smaller than that of the minimizer $\gamma_{\epsilon_k}$.)
Thus $\epsilon_k^{-1} h_{\epsilon_k}$ (resp. $\epsilon_k^{-1} \kappa_{\epsilon_k}$) converge in $C^2$ (resp. $C^0$) to a minimizer of the linearized problem.

\vspace{3mm}

By Proposition \ref{BumpLength}, for $\epsilon < \epsilon_0$ small, the intervals that comprise $\{\epsilon^{-1}h_{\epsilon} > 1\}$ all have length at least $c_0 > 0$. It follows from
the convergence of $\epsilon^{-1} h_{\epsilon}$ (resp. $\epsilon^{-1} \kappa_{\epsilon}$) in $C^2$ (resp. $C^0$) to a minimizer of the linear problem and Proposition \ref{LinearMinimizers}
that the sets $\{\epsilon^{-1}h_{\epsilon} > 1\}$ consist of exactly one interval that converges in the Hausdorff distance to the separation
interval for a linear minimizer. This completes the proof.
\end{proof}

%------------

\section{Appendix}
\subsection{Derivation of Euler-Lagrange Equation}
For a curve on the sphere of length $2\pi$, let $\gamma$ be a unit-speed parametrization. We define the unit normal $N$ to the cone over $\gamma$ by
$$N = \gamma \times \gamma'.$$
Easy computations give that
\begin{equation}\label{CurvatureRelation}
\gamma'' = -\gamma + \kappa N, \qquad \gamma'' \cdot N = \kappa.
\end{equation}

If $\bar{\gamma}$ is a curve on the sphere, but not parametrized by arc length, the above formula can be used to derive the geodesic curvature $\bar{\kappa}$ of $\bar{\gamma}$:
\begin{equation}\label{CurvatureFormula}
\bar{\kappa} = \frac{\bar{\gamma}'' \cdot (\bar{\gamma} \times \bar{\gamma}')}{|\bar\gamma'|^3}.
\end{equation}

Now let $\psi \in W^{2,2}(\mathbb{S}^1)$ satisfy $\psi \cdot \gamma = 0$. Using  (\ref{CurvatureFormula}), we compute the geodesic curvature $\kappa_{\delta}$
of the projection of $\gamma + \delta \psi$ to $\mathbb{S}^2$:
$$\kappa_{\delta} = \kappa + \delta\bigl(\psi'' \cdot N - 3\kappa(\psi' \cdot \gamma') + \gamma'' \cdot (\psi \times \gamma' + \gamma \times \psi')\bigr) + O(\delta^2).$$ 
Using the vector identity ${\bf a} \cdot ({\bf b} \times {\bf c}) = {\bf b} \cdot ({\bf c} \times {\bf a})$ we conclude that
$$ \kappa_{\delta}^2\,ds_{\delta} = \left(\kappa^2 + 2\delta\left(\kappa N \cdot \psi'' - \frac{3}{2}\kappa^2 \gamma' \cdot \psi' + \kappa N \cdot \psi \right) + O(\delta^2)\right)\,ds, $$
where $ds_\delta$ is the length element of $\gamma_\delta:=\frac{\gamma+\delta\psi}{|\gamma+\delta \psi|}$, that is 
$$ds_{\delta} = (1 + \delta \gamma' \cdot \psi' + O(\delta^2))ds.$$
Thus, the first-order change in $\int \kappa^2\,ds$ is given by
$$\int_{\mathbb{S}^1} \left(\kappa N \cdot \psi'' - \frac{3}{2}\kappa^2 \gamma' \cdot \psi' + \kappa N \cdot \psi \right)\,ds.$$

%%%%%

\subsection{Lagrange Multiplier}
We can remove the length constraint $\int_{\mathbb{S}^1} \kappa N \cdot \psi \,ds = 0$ by introducing a Lagrange multiplier $\lambda$.

Let $\gamma$ be a unit-speed minimizer of $F$ in $\tilde Y_{\epsilon}$. Let $J$ be the contact set $\{\gamma \cdot e_3 = \epsilon\}$ of the minimizer with the obstacle. Let
$$Z  := \{\varphi - (\varphi \cdot \gamma) \gamma: \varphi \in W^{2,2}(\mathbb{S}^1; \mathbb{R}^3)\}$$
be the subspace of $W^{2,2}(\mathbb{S}^1; \mathbb{R}^3)$ that is tangent to $\mathbb{S}^2$ on $\gamma$. Let $\psi_{\varphi}$ denote an element of $Z$ generated by $\varphi$.
Finally, let
$$K := \{\psi \in Z: \psi \cdot e_3 \geq 0 \text{ on } J\}.$$

Note that $K$ is convex. Furthermore, $K - K= Z$. Indeed, since $\|\gamma \cdot e_3\|_{C^1(\mathbb{S}^1)} \leq C\epsilon$, we have for $\varphi = Ae_3$ with $A$ large that $\psi_{\varphi} \cdot e_3 \sim A >> 1$. Thus, for any
$\tilde{\psi} = \psi_{\tilde{\varphi}} \in Z$ we have for $A$ large that $(\psi_{\varphi} - \psi_{\tilde{\varphi}}) \cdot e_3 = \psi_{\varphi - \tilde{\varphi}} \cdot e_3 > 0$.

For $\psi \in Z$, let $\kappa_{\psi}$ be the geodesic curvature of $\frac{\gamma + \psi}{|\gamma+\psi|}$, with arc length parameter $s_{\psi}$, and let
$$G(\psi) := \int_{\mathbb{S}^1} \kappa_{\psi}^2\,ds_{\psi},\qquad 
H(\psi) := \int_{\mathbb{S}^1} ds_{\psi} - 2\pi.$$
By construction, $0\in Z$ is a local minimizer of the variational problem
$$\min\{G(\psi): \psi \in K,\, H(\psi) = 0\},$$
and by the above computations both $G$ and $H$ are Fr\'{e}chet differentiable in a neighborhood of $0$.
Thus, we can apply \cite[Theorem 4]{O} (taken from \cite{DMO}) to conclude that there exists some $\lambda \in \mathbb{R}$ such that
$$DG(0)[\psi] + \lambda DH(0)[\psi] \geq 0\qquad \forall\,\psi \in K.$$
By the above computations, this becomes
$$\int_{\mathbb{S}^1} \left(\kappa N \cdot \psi'' - \frac{3}{2} \kappa^2 \gamma' \cdot \psi' + (1 + \lambda)\kappa N \cdot \psi\right)\,ds \geq 0$$
for all $\psi \in K$ and some $\lambda \in \mathbb{R}$.

%%%%%%%%%%%%%%%%%%%%%%%%%%%%%%%%%%%%%%%%%%%%%%%%%%%%%%%%%%%%%%%%%%%%%%%%%%%%%%%%%%%%%%%%
\section*{Acknowledgements}
Both authors were supported by the ERC grant ``Regularity and Stability in Partial Differential Equations (RSPDE)''. C. Mooney was supported in part by National Science Foundation grant DMS-1501152.
We would like to thank Francesco Maggi for many valuable discussions.

%%%%%%%%%%%%%%%%%%%%%%%%%%%%%%%%%%%%%%%%%%%%%%%%%%%%%%%%%%%%%%%%%%%%%%%%%%%%%%%%%%%

%%%%%%%%%%%%%%%%%%%%%%%%%%%%%%%%%%%%%%%%%%%%%%%%%%%%%%%%%%%%%%%%%%%%%%%%%%%%%%%%%%%%

\end{document}